\documentclass[11pt,reqno]{amsart}
\usepackage[]{amsmath,amssymb,amsfonts,latexsym,amsthm,enumerate}
\usepackage{amssymb,bbm,mathrsfs,tikz}
\usepackage{enumerate,graphicx,paralist,booktabs}
\usepackage[noadjust]{cite}

\numberwithin{equation}{section}
\usepackage[colorlinks=true, pdfstartview=FitV, linkcolor=blue,
  citecolor=blue, urlcolor=blue,pagebackref=false]{hyperref}


\newtheorem{theorem}{Theorem}[section]
\newtheorem*{theorem*}{Theorem}
\newtheorem{lemma}[theorem]{Lemma}
\newtheorem{claim}[theorem]{Claim}
\newtheorem{proposition}[theorem]{Proposition}
\newtheorem{observation}[theorem]{Observation}
\newtheorem{corollary}[theorem]{Corollary}

\theoremstyle{definition}{

\newtheorem*{definition*}{Definition}

\newtheorem*{question*}{Question}
\newtheorem*{example*}{Example}
\newtheorem*{examples*}{Examples}
\newtheorem{remark}[theorem]{Remark}
\newtheorem*{remark*}{Remark}

}

\newcommand{\semic}{{\textsc{sc}}}
\DeclareMathOperator{\tr}{tr}
\DeclareMathOperator{\xor}{\triangle}

\newcommand{\abbr}[1]{{\sc{\lowercase{#1}}}}

\newcommand{\bA}{\mathbf{A}}
\newcommand{\ba}{\mathbf{a}}
\newcommand{\bB}{\mathbf{B}}

\newcommand{\bD}{\mathbf{D}}
\newcommand{\bDelta}{\mathsf{\Delta}}

\newcommand{\bLambda}{\boldsymbol{\Lambda}}
\newcommand{\bM}{\mathbf{M}}

\newcommand{\brho}{\boldsymbol{\rho}}

\newcommand{\bX}{\mathbf{X}}
\newcommand{\bY}{\mathbf{Y}}

\newcommand{\CC}{\mathbb{C}}

\newcommand{\EE}{\mathbb{E}}

\newcommand{\NN}{\mathbb{N}}

\newcommand{\PP}{\mathbb{P}}

\newcommand{\RR}{\mathbb{R}}

\newcommand{\ZZ}{\mathbb{Z}}

\newcommand{\cL}{\mathcal{L}}

\newcommand{\cS}{\mathcal{S}}

\newcommand{\sfC}{\mathsf {C}}
\newcommand{\sfD}{\mathsf {D}}
\newcommand{\sfE}{\mathsf {E}}

\newcommand{\sfG}{\mathsf {G}}
\newcommand{\sfH}{\mathsf {H}}

\newcommand{\sfV}{\mathsf{V}}

\newcommand{\sfX}{\mathsf {X}}
\newcommand{\sfY}{\mathsf{Y}}

\newcommand{\sfm}{\mathsf{m}}

\newcommand{\sE}{\mathscr{E}}

\newcommand{\sG}{\mathscr{G}}

\newcommand{\one}{\mathbbm{1}}
\newcommand{\PPo}{\operatorname{Po}}
\newcommand{\Bin}{\operatorname{Bin}}

\newcommand{\supp}{\operatorname{supp}}

\renewcommand{\epsilon}{\varepsilon}

\author{Amir Dembo}
\address{Amir Dembo\hfill\break
Department of Mathematics\\ Stanford University\\ Sloan Hall \\
Stanford, CA 94305, USA.}
\email{amir@math.stanford.edu}

\author{Eyal Lubetzky}
\address{Eyal Lubetzky\hfill\break
Courant Institute 
\\ New York University\\
251 Mercer Street\\ New York, NY 10012, USA.}
\email{eyal@courant.nyu.edu}

\author{Yumeng Zhang}
\address{Yumeng Zhang\hfill\break
Department of Statistics\\ Stanford University\\ 390 Serra Mall \\
Stanford, CA 94305,~USA.}
\email{zym3008@gmail.com}

\title[Empirical spectral distributions of sparse random graphs]
{Empirical spectral distributions\\ of
sparse random graphs}

\begin{document}

\begin{abstract}
We study the spectrum of a random multigraph with
a degree sequence $\bD_n=(D_i)_{i=1}^n$ and average degree $1 \ll \omega_n \ll n$, generated by the configuration model, and also the spectrum of the analogous random simple graph.
We show that, when the empirical spectral distribution (\abbr{ESD}) of $\omega_n^{-1} \bD_n $ converges weakly to a limit $\nu$, under mild moment assumptions (\emph{e.g.}, $D_i/\omega_n$ are i.i.d.\ with a finite second moment),
the \abbr{ESD} of the normalized adjacency matrix converges in probability to
$\nu\boxtimes \sigma_{\semic}$, the free multiplicative convolution of $\nu$ with the semicircle law. Relating this limit
with a variant of the Marchenko--Pastur law yields the continuity of its
density (away from zero), and an effective procedure for determining its support.

Our proof of  convergence is based on a 
coupling between the random simple graph and multigraph with the same degrees, which might be of independent interest. We further construct and rely on a coupling of the multigraph to an inhomogeneous Erd\H{o}s-R\'enyi graph with the target \abbr{ESD}, using three intermediate random graphs, with a negligible fraction of edges modified in each step.
  \end{abstract}


{\mbox{}\vspace{-1.85cm}
\maketitle
}
\vspace{-0.48cm}

\section{Introduction}\label{secIntro}

We study the spectrum of a  random multigraph $\sfG_n=([n],\sfE_n)$ of $n$ vertices 
of degrees $\{D_i^{(n)}\}_{i=1}^n$,  constructed by the configuration model, where 
the \emph{even} 
\begin{equation}\label{eq-omega-choice} 
\sum_{i=1}^n D_i^{(n)} = 2|\sfE_n| := n \omega_n (1+o(1)) \,, 
\end{equation}
is assumed to be such that 
\begin{equation}\label{eq-assumption-En}  
\omega_n \to\infty\,,\qquad \omega_n = o(n) \,.
\end{equation}
Specifically, setting $[n]=\{1,2,\ldots,n\}$, equip each vertex $i \in [n]$ with $D_i^{(n)}$ half-edges, 
whereby the edge set $\sfE_n$ results from a uniformly chosen perfect matching 
of the $2|\sfE_n|$ half-edges. 
The uniformly chosen  \emph{simple} graph $\sG_n=([n],\sE_n)$
with the degrees $D_i^{(n)}$ --- assuming of course that this degree sequence 
is graphical
(i.e., there exist simple graphs with these degrees) ---
is similarly described via a uniform perfect matching of half-edges, 
subject to the constraint of 
having neither self-loops nor multiple edges.

Our study of the spectrum of the adjacency matrix $\bA_{\sfG_n}$ of the multigraph $\sfG_n$,
proceeds through a sequence of couplings, relating it to certain ``band'' matrices, with independent albeit non-identically-distributed entries (adjacency matrices of Erd\H{o}s-R\'enyi \emph{inhomogeneous} random graphs). Various spectral features of the latter will then be derived using the powerful tools that have been developed in the last few decades in random matrix theory and free probability.

In Proposition \ref{prop:coupling} we further provide a novel
coupling 
of $\sG_n$ and $\sfG_n$, which may be of independent interest. Utilizing this coupling we deduce
that the uniformly chosen random \emph{simple} graph $\sG_n$, 
satisfying the same degree assumptions as $\sfG_n$, will also have the same limiting spectrum.

For random \emph{regular} graphs---the case of $D_i^{(n)}=d_n$ for all $i$--- it was shown by Tran, Vu and Wang~\cite{TVW13} (extending a previous result of~\cite{DuPa12}) that, whenever $ d_n\gg 1$, the empirical spectral distribution (\abbr{ESD}, defined for a symmetric matrix $\bA$ with eigenvalues $\lambda_1\geq\ldots\geq\lambda_n$ as $\cL^{\bA} = \frac1n\sum_{i=1}^n\delta_{\lambda_i}$)
of the normalized matrix  $\hat \bA_{\sfG_n}= \frac1{\sqrt{d_n}} \bA_{\sfG_n}$ converges weakly, in probability, to $\sigma_{\semic}$, the standard semicircle law (with support $[-2,2]$).

The non-regular case with $|\sfE_n|=O(n)$ has been studied by Bordenave and Lelarge~\cite{BoLa10} when the graphs $\sfG_n$ converge in the Benjamini--Schramm sense, translating in the above setup to having $\{D_i^{(n)}\}$ that are i.i.d.\ in $i$ and uniformly integrable in $n$. 
The existence and uniqueness of the limiting \abbr{ESD} was obtained in~\cite{BoLa10} by relating this \abbr{ESD} to a recursive distributional equation --- arising from the Galton--Watson trees that correspond to the local neighborhoods in $\sfG_n$ --- and showing that this equation has a unique fixed point. See also, e.g.,~\cite{Bordenave17,CoSa18,Salez16} and the references therein, for the analysis of the limiting spectrum at $\lambda=0$ for Erd\H{o}s--R\'enyi graphs of constant average degree.
Note that (a) this approach relies on the locally-tree-like structure of the graphs, and is thus tailored for low (at most logarithmic) degrees; and (b) very little is known on this limit, even in seemingly simple settings such as when all degrees are either 3 or 4.

At the other extreme, when $|\sfE_n|$ diverges polynomially with $n$ (whence the tree approximations are invalid), the trace method---the standard tool for establishing the convergence of the \abbr{ESD} of an Erd\H{o}s--R\'enyi random graph to $\sigma_{\semic}$---faces the obstacle of non-negligible dependencies between edges in the configuration model (the trace method can handle dependencies, but here 
$n^{-1} \tr  ( \, (\EE \hat \bA_{\sfG_n})^{2k} ) \asymp \omega_n^k$, thus the  
precise cancellations of many diverging terms are needed for it to work; such cancellations 
are very difficult to attain in the presence of dependencies).

\subsection{Limiting ESD as a free multiplicative convolution }\label{sec:intro-1} 
 
Our assumptions on the triangular sequence $\{D_i^{(n)}\}$ 
of degrees are that \eqref{eq-assumption-En} holds,
and in addition, for $\omega_n$ satisfying \eqref{eq-omega-choice}, 
 the normalized degrees $\hat D_i^{(n)} = D_i^{(n)}/\omega_n$ satisfy that
\begin{equation}\label{eq-Dn-moment}
\{ \hat D_{U_n}^{(n)} \}
\quad\mbox{ is uniformly integrable with }\quad
\EE [( \hat D_{U_n}^{(n)} )^2] = o(\sqrt{n/\omega_n})\,,
\end{equation}
where $U_n$ is uniformly chosen in $\{1,\ldots,n\}$. Let
\[ \hat \bA_{\sfG_n} :=\omega_n^{-1/2} A_{\sfG_n}\qquad\mbox { and }\qquad \hat \bLambda_n = \mathrm{diag}(\hat D^{(n)}_1,\ldots,\hat D^{(n)}_n)\,.\]
Call a degree sequence $\{D_i^{(n)}\}$ \emph{graphical} if for every $n$ there exists a simple graph $\sG_n$ with such degrees (equivalently,
the criterion of the Erd\H{o}s--Gallai theorem \cite{EG60} is met).
\begin{theorem}\label{thm:1}
Let $\{D_i^{(n)}\}_{i=1}^n$ be a degree sequence satisfying~\eqref{eq-omega-choice}--\eqref{eq-Dn-moment}, and further suppose that the \abbr{ESD} $ \cL^{\hat \bLambda_n} $ converges weakly to a limit~$\nu_{\hat D}$. 
\begin{enumerate}[(a)]
	\item The \abbr{ESD} $\cL^{\hat \bA_{\sfG_n}}$ corresponding to the multigraph $\sfG_n=([n],\sfE_n)$ with degrees $\{D_i^{(n)}\}_{i=1}^n$ (generated via the configuration model),
 converges weakly, in probability, to $\nu_{\hat D} \boxtimes \sigma_{\semic}$.

	\item If $\{D_i^{(n)}\}$ is graphical then  
	the same convergence holds for the \abbr{ESD} $\cL^{\hat \bA_{\sG_n}}$ corresponding to a uniformly chosen simple graph $\sG_n=([n],\sE_n)$ with this degree sequence. 
\end{enumerate}

\end{theorem}

In the above theorem, the \emph{free multiplicative convolution} of a symmetric probability measure $\psi$ and a probability 
measure $\varphi$ on $\RR_+$
with $\varphi,\psi\neq \delta_0$, denoted $\varphi \boxtimes \psi$, is  the unique probability measure such that $S_{\varphi\boxtimes \psi}(z) = S_\varphi(z) S_{\psi}(z)$ for $z$ in the common domain of the corresponding $S$-transforms (see~\cite[Thm.~7]{ArPe09}, extending the definition of 
$\varphi\boxtimes \psi$ from~\cite{BeVo93} and~\cite{Voiculescu87} in
case both $\varphi$, $\psi$, are of bounded support and non-zero mean). 
To define the $S$-transform,
recall that the Cauchy--Stieltjes transform of a probability measure $\mu$
on $\RR$, uniquely determining it, is 
$G_\mu(z) := \int [t-z]^{-1}d\mu(t)$. For $\varphi$ as above, the related 
\begin{equation}\label{dfn:m-trans}
m_\varphi(z):= 
z^{-1} G_\varphi(z^{-1}) -1 =
\int \frac{zt}{1-zt} d \varphi(t)
 \,,
\end{equation}
is invertible as a formal power series in $z \in \CC_+$,
and the $S$-transform is defined as 
\begin{equation}\label{eq:S-trans}S_\varphi(w):=(1+w^{-1}) m_\varphi^{-1}(w)\quad\mbox{ for $w\in m_\varphi(\CC_+)$}	
\end{equation}
(cf., \emph{e.g.},~\cite[Prop. 1]{ArPe09}).
Following the extension in~\cite{RaSp07}
of the $S$-transform to measures of zero mean and
finite moments of all order, the $S$-transform is similarly defined for $\psi$ as above in~\cite[Thm.~6]{ArPe09}. In particular, with $\sigma_{\semic}$ being symmetric and $\nu_{\hat D} \ne \delta_0$ supported on $\RR_+$, the measure $\nu_{\hat D} \boxtimes \sigma_{\semic}$ is well-defined.

\begin{figure}[t]
\centering
\includegraphics[width=.49\textwidth]{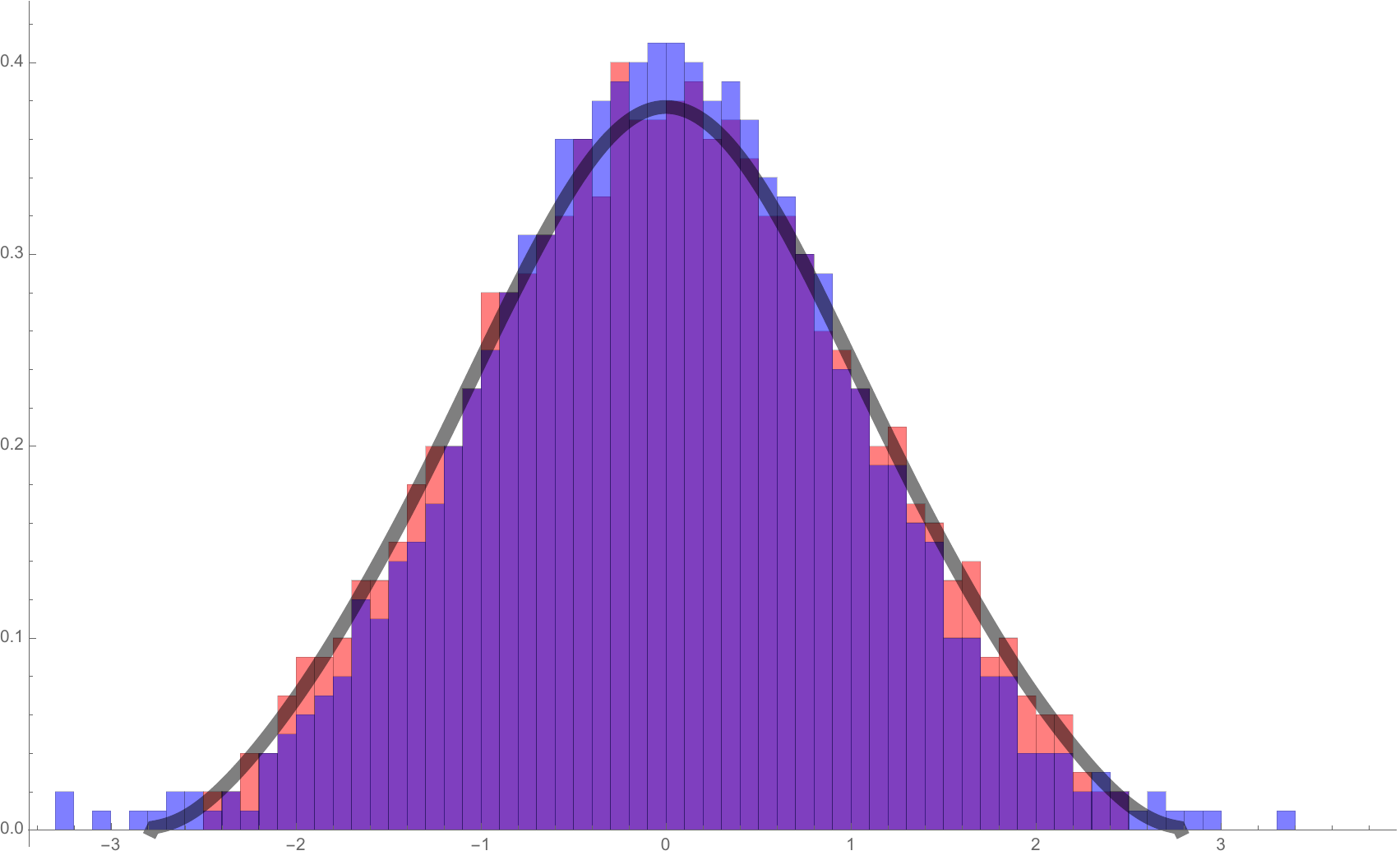}
\includegraphics[width=.49\textwidth]{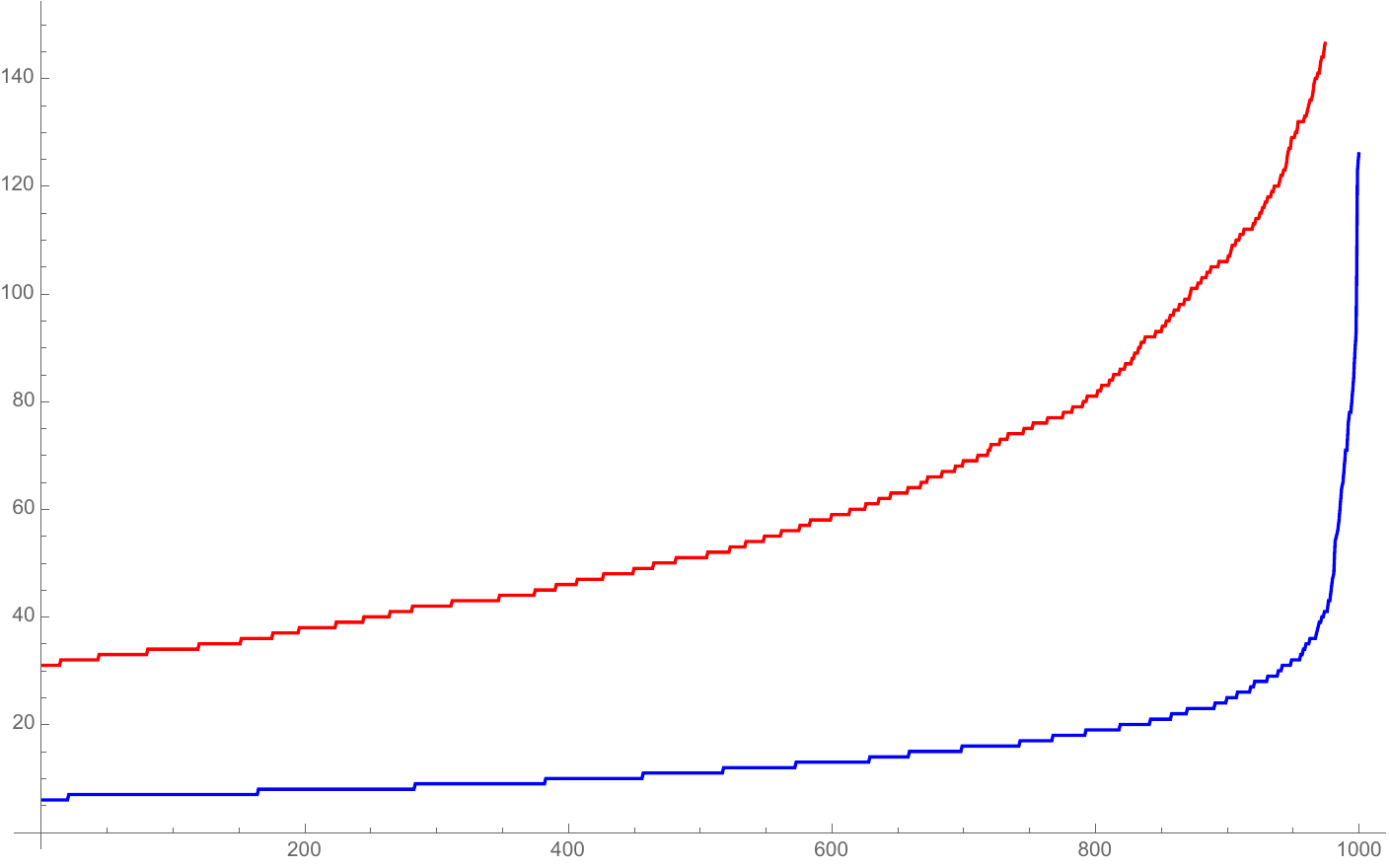}
\caption{Spectra of two random multigraphs on $n=1000$ vertices with different degree sequences $\{D_i\}$. In red, $D_i=[\tau_i\sqrt{n}]$ for all $i$, and in blue, $D_i=[\tau_i\log{n}]$ for $i <n-\sqrt{n}$ and $D_i=[\tau_i \sqrt{n}]$ for $i\geq n-\sqrt{n}$, with  $\tau_i\sim1+\mathrm{Exp}(1)$ i.i.d.\ (right plot). The limiting law for the \abbr{esd}, shown by Theorem~\ref{thm:1} to be $\nu_{\hat D}\boxtimes \sigma_{\semic}$, is plotted in black (left plot).}
\label{fig:c1-tmix}
\end{figure}

\begin{corollary}\label{cor-iid}
Let $\{\hat{D}_i^{(n)} : 1\leq i \leq n\}$ be i.i.d.\ for each $n$,
such that $\EE\hat D_1^{(n)}=1$, $\sup_n \EE[ (\hat D_1^{(n)})^2] < \infty$,
and the law of $\hat {D}_1^{(n)}$ converges weakly to some $\nu_{\hat D}$. 
For $\omega_n\to\infty$ such that $\omega_n = o(n)$, let $\sfG_n$ denote
the uniform multigraph of degrees $D_i^{(n)} = [ \omega_n \hat D_i^{(n)}]$
(modifying $D_n^{(n)}$ by one if needed for an even sum). 
Further, for any integers $\bar d_n=o(n)$ with $\omega_n = o(\bar d_n)$,
the truncated  degrees $[\omega_n \hat D_i^{(n)}] \wedge \bar d_n$ 
are graphical \abbr{whp} (after increasing the minimal degree by one, 
if needed, for an even sum). Denoting by
$\sG_n$ the uniform simple graph, both \abbr{ESD}s $\cL^{\hat \bA_{\sfG_n}}$ 
and $\cL^{\hat \bA_{\sG_n}}$ converge weakly, in probability, to
$\nu_{\hat D} \boxtimes \sigma_{\semic}$.
\end{corollary}

\begin{remark}\label{rem:goe}
The reason for the appearance of
$\nu_{\hat D} \boxtimes \sigma_{\semic}$
in our context is due to the fact that it is the limiting \abbr{ESD} 
of $\bB_n := \hat \bLambda_n^{1/2} \bX_n \hat \bLambda_n^{1/2}$ when 
$\max_i \hat{D}_i^{(n)} = O(1)$ and $\bX_n$ is a standard \abbr{GOE} random matrix. 
Indeed, as its name suggest, the free multiplicative convolution 
$\varphi \boxtimes \psi$ is the law of the 
product $a b$ of free, bounded, random non-commutative operators 
$a$ of law $\varphi$ and $b$ of law $\psi$ (cf.~\cite[Defn. 5.2.1, 5.2.3, 5.3.1, 5.3.28]{AGZ10} 
for the precise meaning of all this). This extends to the 
limiting \abbr{ESD} for the product of \emph{asymptotically free} matrices: 
two sequences $\bX_n,\bY_n$ of random self-adjoint, matrices are asymptotically free if
$\EE[\tr_n(f_1(\bX_n) g_1(\bY_n)\cdots f_k(\bX_n) g_k(\bY_n))]=o(1)$
for the normalized trace $\tr_n(\cdot)=\frac{1}{n} \tr(\cdot)$ and
any collections $(f_i)_{i=1}^k$ and $(g_i)_{i=1}^k$ of polynomials (with $k$ fixed) that satisfy $\EE[\tr_n(f_i(\bX_n))]=o(1)$ and $\EE[\tr_n(g_i(\bY_n))]=o(1)$ for all $1\leq i\leq k$
(see~\cite[Defn.~5.4.1]{AGZ10} or~\cite[\S2.5]{Tao12}).
It is known that the \abbr{GOE} $\bX_n$ is asymptotically 
free of any uniformly bounded diagonal
$\hat \bLambda_n$ (see, \emph{e.g.},~\cite[Theorem~5.4.5]{AGZ10}),
which in turn implies that 
$\nu_{\hat D} \boxtimes \sigma_{\semic}$ is 
the weak limit of the \abbr{esd} for the
random matrices $\bB_n
$ (the spectral radius of the \abbr{GOE} $\bX_n$ is $O(1)$
with high 
probability,
so by a standard truncation argument we arrive at the 
bounded case of  \cite[Corollary~5.4.11(iii)]{AGZ10}).
\end{remark}

Theorem~\ref{thm:1} and Corollary~\ref{cor-iid} are proved in~\S\ref{s:conv}. This is achieved by first analyzing the \abbr{esd} of the random multigraph $\sfG_n$; the move from multigraphs to simple graphs is achieved via the following proposition, which we prove in~\S\ref{sec:coupling}.

\begin{proposition}\label{prop:coupling}
Fixing graphical degrees $D_1 \ge D_2 \ge \cdots \ge D_n$, let $\sfG_n$ and $\sG_n$ be 
the corresponding random multigraph and uniform simple graph, respectively.
There exists a coupling $\mu$ between the matchings
which yield $\sfG_n$ and $\sG_n$ so the number $\bDelta_n \le 2|\sfE_n|$
of half-edges whose matching links are different between the two graphs, 
 satisfies
\begin{align}\label{e:deltaEn}
 \EE_\mu [  \bDelta_n (\bDelta_n-1) ]  &\le  4\sum_{i=1}^{n-1} \sum_{j=i+1}^{i+D_i} (2 D_i D_j - D_i - D_j)  
 \,. 
\end{align}	
\end{proposition}	
\begin{remark} A crude, yet already useful, upper bound on the \abbr{rhs} of~\eqref{e:deltaEn} is
\begin{equation}
	8\sqrt{2|\sfE_n|} \sum_{i=1}^n D_i^2 \label{eq:deltaEn-simple}\,.
\end{equation}
(Indeed, $\big(\sum_{j=i+1}^{i+D_i} D_j\big)^2 \le D_i \sum_{j=1}^n D_j^2 $ by Cauchy--Schwarz for any $i\in[n]$; thus, again by Cauchy--Schwarz, the \abbr{rhs} of~\eqref{e:deltaEn} is at most $8 (\sum_i D_i^2 )(\sum_i D_i)^{\frac12}$, and $\sum_i D_i = 2|\sfE_n|$.)
In general, the \abbr{rhs} of \eqref{e:deltaEn} can be replaced by any bound
on the expected number of pairs of half-edges $e \ne f$ on a which a ``switch'' would yield a non-simple graph. 
\end{remark}

\begin{remark} The proof of Theorem \ref{thm:1}(a) extends to the dense regime, 
where $\omega_n /n$ is bounded below (and above). However, the minimal expected 
edit distance between $\sfG_n$ and $\sG_n$ exceeds the expected number 
$O(\omega_n^2)$ of parallel edges in $\sfG_n$, which in the dense regime is
already $O(|\sfE_n|)$, thereby blocking in the dense regime
our route to Theorem \ref{thm:1}(b) as a consequence of part (a).
Further, our assumption \eqref{eq-Dn-moment} allows 
having a maximal degree that far exceeds $n$ (indeed, prior to truncation 
this happens in the i.i.d.\ setting of Corollary \ref{cor-iid}). Even for specified 
graphical degrees, the number of simple graphs  $\sG_n$ oscillates widely 
as the degrees change, so \eqref{eq-Dn-moment} might not suffice for 
the statement of Theorem \ref{thm:1}(b) to be true in the dense regime. 
Going back to the sparse regime, assumptions \`a la~\eqref{eq-Dn-moment}
have little to do with controlling extreme eigenvalues, or with bringing the
corresponding local law to the celebrated \abbr{GOE}-universality class 
of homogeneous  Erd\H{o}s-R\'enyi graphs. Indeed, one 
must further restrict $\hat \bLambda_n$, in order to have any hope of 
transferring the many fine results on extreme eigenvalues and local laws 
that are available for the \abbr{GOE}, via 
$\bB_n$ of Remark  \ref{rem:goe} to $\hat \bA_{\sfG_n}$.
\end{remark}

\subsection{Properties of the limiting ESD}\label{sec:intro-2}

The next two propositions, proved in~\S\ref{sec:dens},
relate the limiting measure
$\nu_{\hat D} \boxtimes \sigma_{\semic}$ with a
Marchenko--Pastur law, and thereby, via~\cite{SiCh95},
yield its support and density regularity. 

\begin{proposition}\label{lem-mu-Stil}
 For the law $\nu_{\hat D}$  of a nonnegative random variable $\hat D$ with $\EE \hat D = 1$,
 let $\mu_{\textsc{mp}}$ be the Marchenko--Pastur limit 
 (on $\RR_+$)
of the \abbr{esd}  of $n^{-1} \bLambda_n \widetilde \bX_n \widetilde \bX_n^\star \bLambda_n$, in which
 the non-symmetric $\widetilde \bX_n$ has
standard i.i.d.\ complex Gaussian entries and
$\cL^{\bLambda_n} \Rightarrow \nu$ for
non-negative diagonal matrices $\bLambda_n$
and the size-biased $\nu$ with
$\frac{d \nu}{d\nu_{\hat D}} (x) = x$ on $\RR_+$.
The free multiplicative convolution $\mu = \nu_{\hat D} \boxtimes \sigma_{\semic}$
has the Cauchy--Stieltjes transform
\begin{align}
  \label{eq:f-form}
G_\mu(z)=
 -z^{-1} \Big[1 + G_{\widetilde \mu}(z)^2 \Big] \,, \quad \forall z \in \CC_+ \,,
\end{align}
where  $\widetilde \mu$ is
the law of the symmetric $X$ such that $X^2$ is distributed according to $\mu_{\textsc{mp}}$.	
\end{proposition}

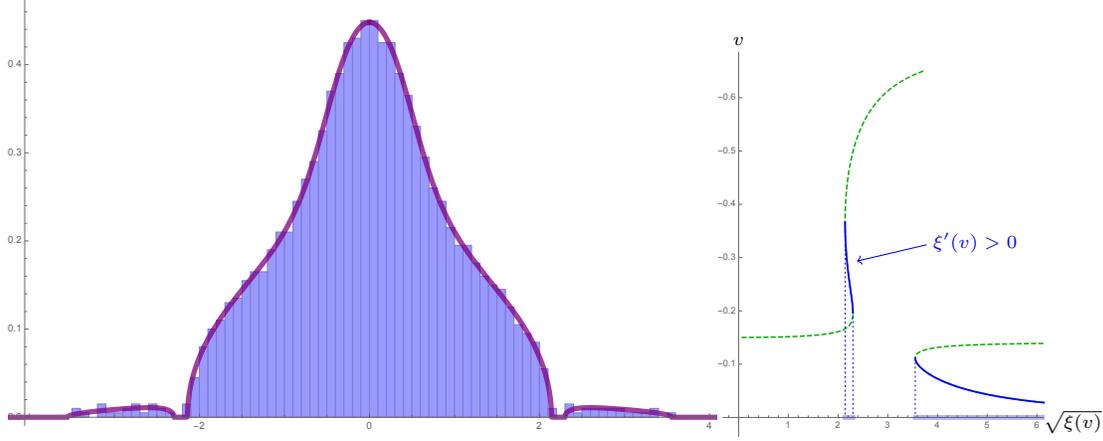
\begin{figure}
\vspace{-0.5cm}
\centering
  \begin{tikzpicture}[scale=0.8]
    \node (fig1) at (0,0) {
	\includegraphics[width=0.65\textwidth]{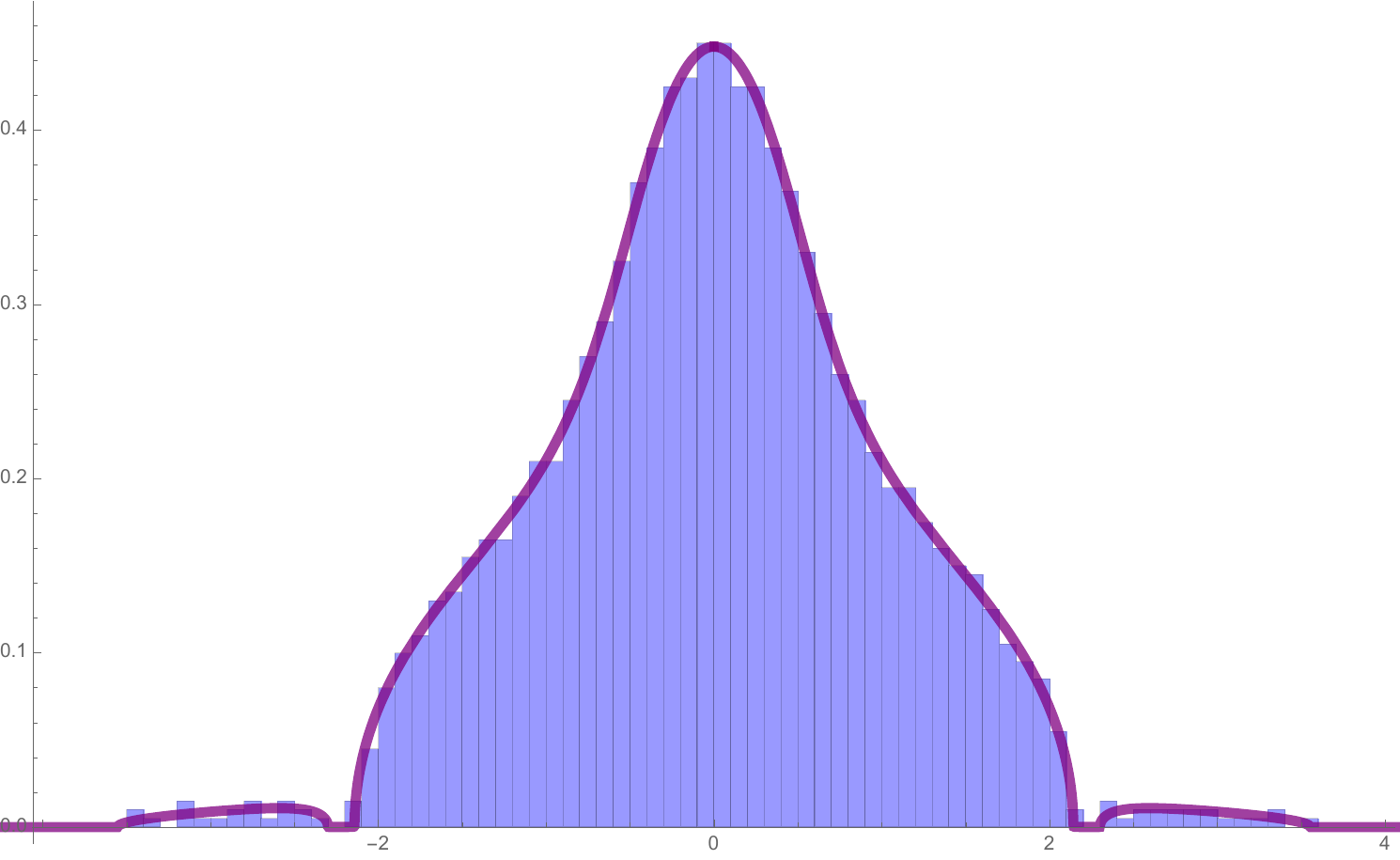}};
    \node (fig2) at (6.9cm,-0.38) {
	\includegraphics[width=0.3\textwidth]{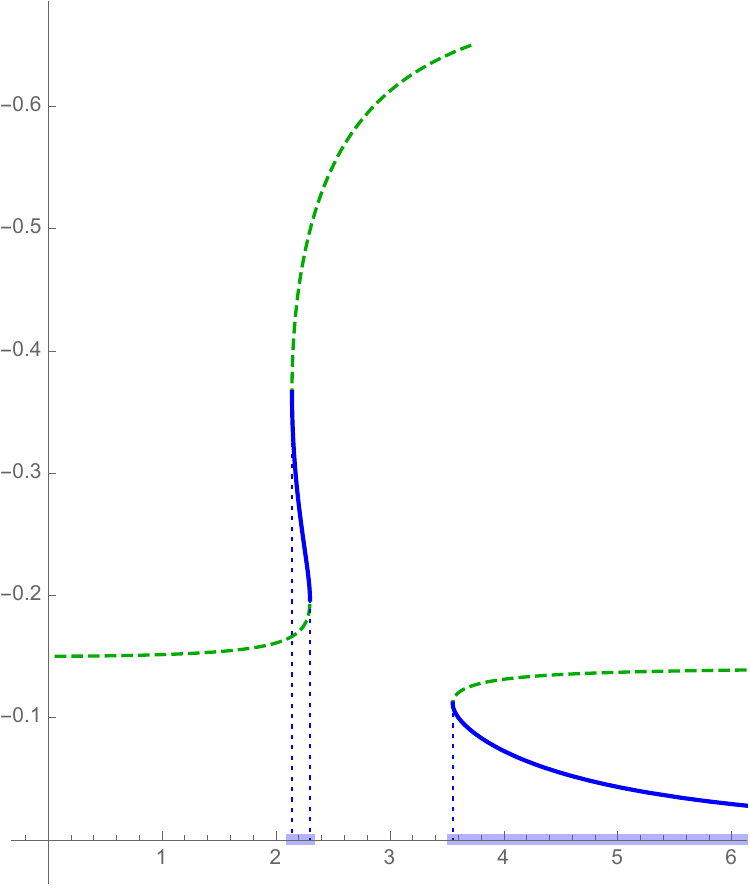}};
    \begin{scope}[shift={(fig2.south west)},x={(fig2.south east)},y={(fig2.north west)}]
      \node[font=\tiny] at (0.09,1) {$v$};
      \node[font=\tiny] at (1.05,0.065) {$\sqrt{\xi(v)}$};
      \node[font=\tiny,color=blue] at (0.77,0.505) {$\xi'(v)>0$};
      \draw[color=blue,thin,->] (0.63,0.5) -- (0.43,0.46) ;
      \end{scope}
  \end{tikzpicture}
\vspace{-0.2cm}
 \caption{Recovering the support of the limiting \abbr{esd}. Left: \abbr{esd} of the random multigraph on $n=1000$ vertices with degrees $\sqrt{n}$, $3\sqrt{n}$, $15\sqrt{n}$ in frequencies $0.5$, $0.49$, $0.01$, resp. Right: $\xi(v)$ from Remark~\ref{rem:pm-support}.}
  \label{fig:support_holes}
\vspace{-0.2cm}
\end{figure}

Recall \cite[Lemma 3.1, Lemma 3.2]{SiCh95}
that $h(z):=G_{\mu_\textsc{mp}}(z)$ is uniformly bounded
on $\CC_+$ away from the imaginary axis, and
\cite[Theorem~1.1]{SiCh95} that $h(z) \to h(x)$
whenever $z \in \CC_+$ converges to $x \in \RR \setminus \{0\}$.
Further, the $\CC_+$-valued function $h(x)$ is
continuous on $\RR \setminus \{0\}$ and the continuous
density
\begin{equation}\label{dfn:pm-density}
\brho_{\textsc{mp}} (x) := \frac{d\mu_{\textsc{mp}}}{dx} = \frac{1}{\pi}
\Im(h(x)) \,,
\end{equation}
is real analytic at any $x \ne 0$ where it is
positive.
The density $\widetilde \brho(x) = |x| \brho_{\textsc{mp}} (x^2)$
of $\widetilde \mu$ inherits these regularity properties.
Bounding $\widetilde \brho$ uniformly and analyzing the
effect of \eqref{eq:f-form} we next make similar conclusions
about the density $\brho(x)$ of $\mu$, now also at $x=0$.
\begin{proposition}\label{prop:density}
In the setting of Proposition~\ref{lem-mu-Stil}, for $x \ne 0$  there is  density
\begin{equation}\label{dfn:rho}
\brho(x) := \frac{d\mu}{dx} = - 2 \Re(h(x^2)) \widetilde \brho(x) \,,
\end{equation}
which is continuous, symmetric, and moreover real analytic where positive. The support of $\mu$
is $\supp(\mu) := \{ x \in \RR : \brho(x) > 0 \} = \supp(\widetilde \mu)$, which
up to the mapping $x \mapsto x^2$ further matches $\supp(\mu_{\textsc{mp}})$.
In addition $\pi \widetilde \brho(x) \le 1 \wedge (2/|x|)$,
$\pi \brho(x) \le (\EE \hat D^{-2})^{1/2} \wedge (4/|x|^{3})$
and
if $\nu_{\hat D}(\{0\})=0$ then
$\mu$ is absolutely continuous (i.e., $\mu(\{0\})=0$).
\end{proposition}

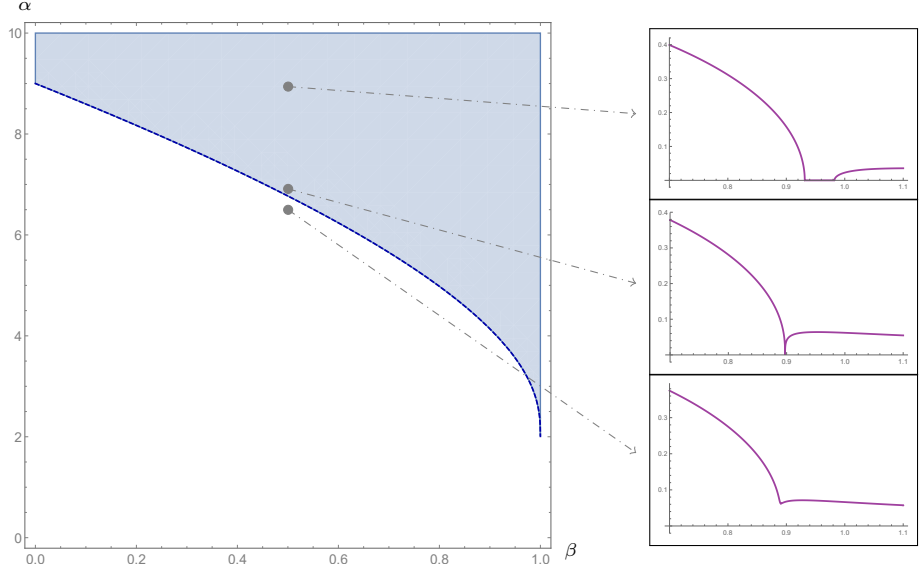
\begin{figure}
\vspace{-0.2cm}
\centering
  \begin{tikzpicture}
    \node (fig1) at (0,0) {
	\includegraphics[width=0.5\textwidth]{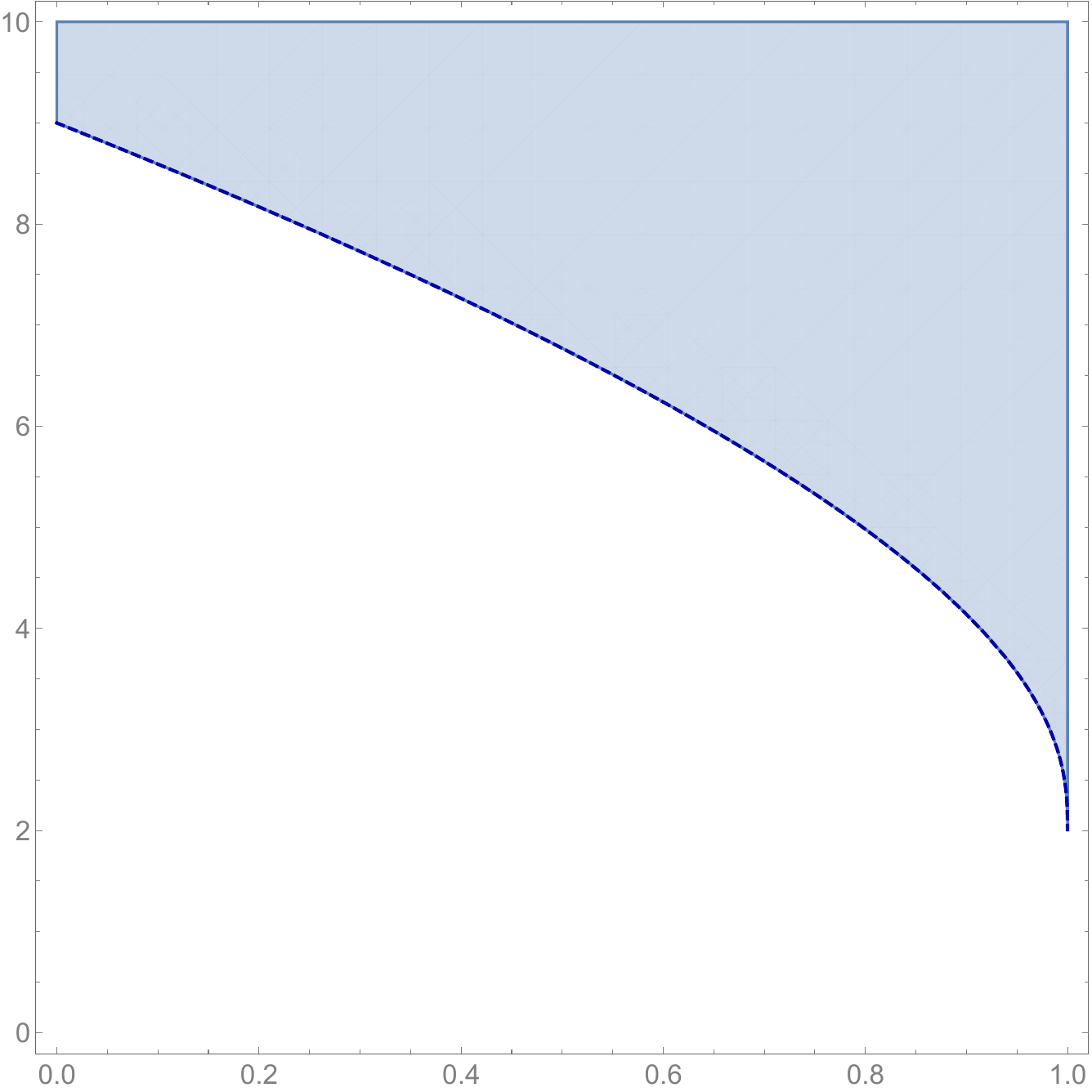}};
    \node (fig2) at (6.7cm,2.4) {
	\fbox{\includegraphics[width=0.23\textwidth]{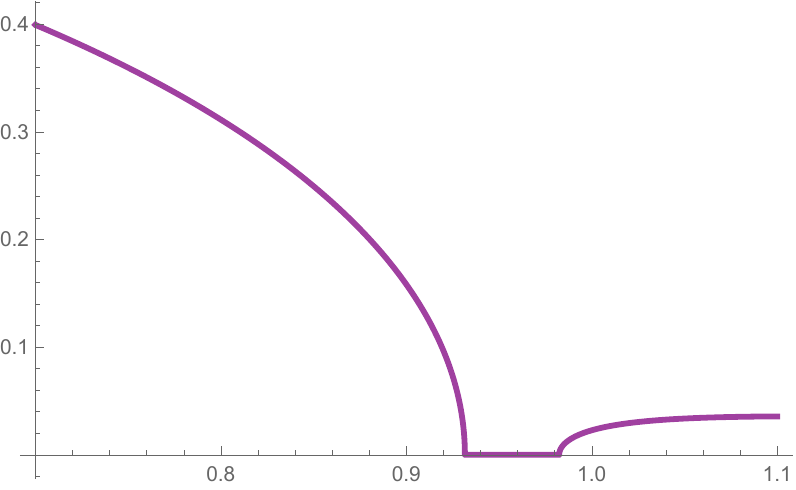}}};
    \node (fig3) at (6.7cm,0.1) {
	\fbox{\includegraphics[width=0.23\textwidth]{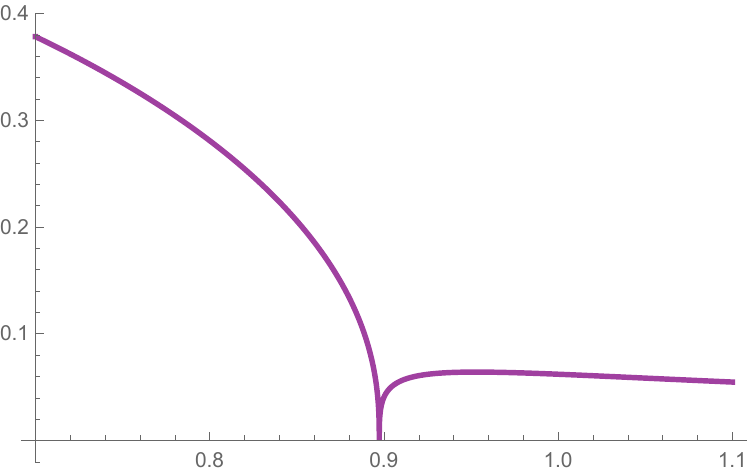}}};
    \node (fig4) at (6.7cm,-2.2) {
	\fbox{\includegraphics[width=0.23\textwidth]{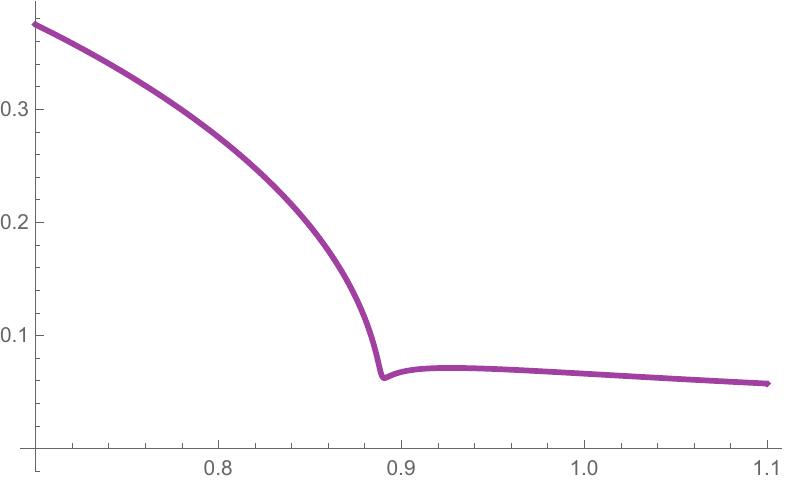}}};
    \begin{scope}[shift={(fig1.south west)},x={(fig1.south east)},y={(fig1.north west)}]
      \node[font=\tiny] at (0.05,1.01) {$\alpha$};
      \node[font=\tiny] at (1.015,0.045) {$\eta$};
      \draw[color=gray,fill=gray] (0.515,0.687) circle (0.008) node[above right] {};
      \draw[color=gray,thin,dashdotted,->] (0.515,0.687) -- (1.13,0.52);
      \draw[color=gray,fill=gray] (0.515,0.65) circle (0.008) node[above right] {};
      \draw[color=gray,thin,dashdotted,->] (0.515,0.65) -- (1.13,0.22);
      \draw[color=gray,fill=gray] (0.515,0.868 ) circle (0.008) node[above right] {};
      \draw[color=gray,thin,dashdotted,->] (0.515,0.868) -- (1.13,0.82);
      \end{scope}
  \end{tikzpicture}
 \caption{Phase diagram for the existence of holes in the limiting \abbr{esd} when $\nu_{\hat D}$ is supported on two atoms $\alpha>\eta>0$ (see Corollary~\ref{cor:two-atoms}). Left: the region~\eqref{eq-support-holes} (where $\supp(\mu)$ is not connected) highlighted in blue. Right: zooming in on the emergence of a hole as $\alpha$ varies at $\eta=\frac12$.}
  \label{fig:support_holes_2atoms}
\vspace{-0.3cm}
\end{figure}

\begin{remark}\label{rem:pm-support}
Recall the unique inverse of $h$ on $h(\CC_+)$ given by
\begin{align}\label{eq-h(z)-inv}
\xi(h) := - \frac{1}{h} + \EE\bigg[\frac{\hat D^2}{1 + h\hat D}\bigg] \,,
\end{align}
namely $\xi(h(z))=z$ on $\CC_+$ (see \cite[Eqn.~(1.4)]{SiCh95});
this inverse
extends analytically to a neighborhood of $\CC_+ \cup \Gamma$ for
$\Gamma:=\{ h \in \RR : h \ne 0, -h^{-1} \in \supp(\hat \nu_D)^c \}$
and~\cite[Theorems~4.1 and~4.2]{SiCh95} show that
$x \in \supp(\mu_{\textsc{mp}})^c$ iff $\xi'(v)>0$
for $v \in \Gamma$, where $v=h(x)$ and $x=\xi(v)$
(thus validating the characterization of $\supp(\mu_\textsc{mp})$
which has been proposed in \cite{MaPa67}).
We show in Lemma~\ref{lem:g-bdd} that $\Re(h(x^2))<0$ everywhere,
hence the behavior of $\brho(x)$ at the soft-edges of $\supp(\mu)$ can be
read from the soft-edges of $\supp(\mu_{\textsc{mp}})$ (as in~\cite[Prop. 2.3]{HHN16}), depicted in Figure~\ref{fig:support_holes_2atoms}.
\end{remark}

\begin{corollary}\label{cor:two-atoms}
Suppose $\nu_{\hat D}$ of mean one is
supported on two atoms $\alpha>\eta>0$. The support $\supp(\mu)$ of
$\mu = \nu_{\hat D} \boxtimes \sigma_{\semic}$ is then disconnected iff
\begin{equation}\label{eq-support-holes}
\alpha > \eta \Big[ \frac{3}{1-(1-\eta)^{1/3}} - 1 \Big] \,.
\end{equation}
Moreover, when~\eqref{eq-support-holes} holds, $\supp(\mu)\cap \RR_+$ consists of exactly two disjoint intervals.
\end{corollary}

\section{Convergence of the ESD's}\label{s:conv}
The proof of Theorem~\ref{thm:1} will use the following standard lemma.
\begin{lemma}\label{lem:truncation}
Let $\{ \bM_{n,r}\}_{n,r\in\NN }$ be a family of matrices of order $n$, define $\mu_{n,r} = \cL^{\bM_{n,r}}$
and $\eta(r) := \limsup_{n\to\infty} \tfrac1n \tr\left((\bM_{n,r} - \bM_{n,\infty})^2\right)$.
Let $\{\mu_r : r\in \NN \}$ denote a family of measures such that
\begin{align}
\label{it-truncation-limit} &\mu_{n,r} \Rightarrow \mu_r\mbox{ as $n\to\infty$ for every $r\in\NN$}\,,\\
\label{it-truncation-tightness} &\mu_{n,\infty}\mbox{ is tight}\,,\\
\label{it-truncation-frob}
&  \eta(r) \to 0 \quad\mbox{ as $r\to\infty$}\,.
 \end{align}
Then the weak limit of $\mu_{n,\infty}$ as $n\to\infty$ exists and equals $\lim_{r\to\infty} \mu_r$.
\end{lemma}
\begin{proof}

Let $\mu_\infty$ be a limit point of $\mu_{n,\infty}$, the existence of which is guaranteed by the tightness assumption~\eqref{it-truncation-tightness}.
A standard consequence of the Hoffman--Wielandt bound (cf.~\cite[Lemma~2.1.19]{AGZ10}) and Cauchy--Schwarz is that for   matrices $\bA$ and $\bB$ of order $n$,
\[ d_{\textsc{bl}}\left(\cL^{\bA} , \cL^{\bB} \right)^2\leq  \frac1n \tr \left( (\bA-\bB)^2\right)\,,\]
where $d_{\textsc{bl}}$ is the bounded-Lipchitz metric on the space $M_1(\RR_+)$ of probability measures on $\RR_+$
(see the proof of~\cite[Theorem~2.1.21]{AGZ10}).
Thus, by~\eqref{it-truncation-limit} and the triangle-inequality for $d_{\textsc{bl}}$, it follows that
\[ \eta(r) \geq d_{\textsc{bl}}(\mu_\infty,\mu_r)^2\,.
\]
Consequently, $\mu_r \to \mu_\infty$ as $r\to\infty$, from which the uniqueness of $\mu_\infty$ also follows.
\end{proof}

\begin{proof}[\textbf{\emph{Proof of Theorem~\ref{thm:1}}}]
{\bf Step I} will reduce the proof to treating the
\emph{single-adjacency} matrix $\bA_n$ of $\sfG_n$, where multiple copies of an
edge/loop are replaced by a single one (that is, $\bA_n = \bA_{\sfG_n} \wedge 1$
entry-wise), and further the collection $\{\omega_n^{-1} D^{(n)}_i\}$
is a fixed finite set $\cS$. Scaling $\hat{\bA}_n := \omega_n^{-1/2} \bA_n$
we rely in {\bf Step II}, on Proposition~\ref{prp:coup} to replace
the limit points of $\cL^{\hat \bA_n}$ by those of $\cL^{\omega_n^{-1/2} \widetilde \bA_n}$ for symmetric matrices $\widetilde \bA_n$ of independent Bernoulli entries,
using the moment method in {\bf Step III} to relate the latter to the limit of
$\cL^{\bB_n}$ for the
matrices $\bB_n$ of Remark~\ref{rem:goe}.

\medskip \noindent
{\bf Step I.}
We claim that if
$\cL^{\hat \bA_n} \Rightarrow \mu$
in probability, then the same applies for $\cL^{\hat \bA_{\sfG_n}}$.
This will follow from Lemma~\ref{lem:truncation} with $\bM_{n,r} = \hat \bA_{n}$
and $\bM_{n,\infty} = \hat \bA_{\sfG_n}$ upon verifying that
\begin{equation}\label{eq-xi-limit}
\xi_n :=
\EE \left[\tfrac1n \tr\left(\big(\hat \bA_{\sfG_n}-\hat \bA_{n}\big)^2\right) \right]
\to 0 \,.
\end{equation}
Indeed, condition~\eqref{it-truncation-limit} has been assumed;
condition~\eqref{it-truncation-tightness} follows from the fact that
\[ \frac1{2n} \tr\left(\hat \bA_{\sfG_n}^2\right) \leq \frac1n\tr\left(\big(\hat \bA_{\sfG_n}-\hat \bA_n\big)^2\right) + \frac1n \tr(\hat \bA_n^2) \leq \frac1n \tr\left(\big(\hat \bA_{\sfG_n}-\hat \bA_n\big)^2\right) + \frac{
2|\sfE_n|}{n\omega_n}\,,\]
so in particular $\EE [\frac1n \tr(\hat \bA_{\sfG_n}^2) ]\leq \xi_n + 1+o(1) $, yielding tightness;
and condition~\eqref{it-truncation-frob} holds in probability by~\eqref{eq-xi-limit} and Markov's inequality.
Recall that the stochastic
ordering $X \preceq X'$ denotes that $\PP(X > x) \le \PP(X' > x)$ 
for all $x \in \RR$, or equivalently, that there exists a coupling
of $(X,X')$ such that $\PP(X \le X')=1$.
To establish~\eqref{eq-xi-limit}, 
observe that, for every $i$ and $j$ we have
$(\bA_{\sfG_n})_{i,j} \preceq \Bin(m, q) $ for $m = D_i^{(n)}$ and
$q = (D_j^{(n)}-1_{i=j})/(2|\sfE_n|-1)$, whereas $\Bin(m,q) \preceq Y_\lambda \sim \PPo( \lambda)$ for every $m$ and $\lambda$ such that $1-q \geq e^{-\lambda/m}$. Thus,
\[ \EE\left[(\bA_{\sfG_n}-\bA_n)_{i,j}^2\right] \leq \EE\left[ (Y_\lambda-1)_+^2\right] \leq \lambda^2 \,.\]
Since $q \le \frac{1+o(1)}{\omega_n}$ uniformly over $i,j$, we
take \abbr{wlog} $\lambda = m D_j^{(n)}/|\sfE_n|$,
yielding for $n$ large
\[
\xi_n \leq \frac{2}{n \omega_n} \sum_{i,j=1}^n
\Big[\frac{D_i^{(n)} D_j^{(n)}}{|\sfE_n|}\Big]^2
\le \frac{4\omega_n}{n} \Big[ \frac1n \sum_{i=1}^n (\hat D_i^{(n)})^2 ]
\to 0 \,,
 \]
by our assumption that
$\EE [( \hat D_{U_n}^{(n)} )^2] = o(\sqrt{n/\omega_n})$. Considering hereafter
only single-adjacency matrices, we proceed to reduce
the problem to the case where the variables $\hat D_i^{(n)}$ are all supported
on a finite set. To this end, let $\ell = 2r^2$ for $r \in \NN$ and
\[ \hat D_i^{(n,r)} = \Psi_r(\hat{D}_i^{(n)}) \quad\mbox{ for }\quad \Psi_r(x) := \sum_{j=1}^\ell d_j^{(r)} \one_{\left[d_j^{(r)},d_{j+1}^{(r)}\right)}(x)\,,
 \]
where $0=d_1^{(r)}<\ldots<d_{\ell+1}^{(r)}$ are continuity points of $\nu_{\hat D}$ of interdistances in $[\frac1{2r},\frac1r]$,
which are furthermore in $\epsilon_r\ZZ$ for some irrational $\epsilon_r>0$.
Let
\[ D_i^{(n,r)} = \omega_{n,r} \hat{D}_i^{(n,r)} \in \ZZ_+ \quad\mbox{ for }\quad \omega_{n,r} := \frac{[\epsilon_r \omega_n ] }{\epsilon_r}\,,\]
 possibly deleting one half-edge from $D_n^{(n,r)}$ if needed to make $\sum_{i=1}^n D_i^{(n,r)}$ even.

\begin{observation}\label{obs:coupling}
Let $\{d_i\}_{i=1}^n,\{d'_i\}_{i=1}^n$ be degree sequences with $d'_i\leq d_i$, and let $\sfG$ be a random multigraph with degrees $\{d_i\}$ generated by the configuration~model. Generate $\sfH$ by
(a) marking a subset of $d'_i$ half-edges of vertex $i$ \textsc{blue}, chosen independently of the matching 
that generated $\sfG$; (b) retaining every edge that has two \textsc{blue} endpoints; and (c) adding an independent uniform matching on all other \textsc{blue} half-edges. Then $\sfH$ has the law of the random multigraph with degrees $\{d'_i\}$ generated by the configuration~model.
\end{observation}
\begin{proof}
Since the configuration model matches the half-edges in $\sfG$ via a uniformly chosen perfect matching, and the coloring step (a) is performed independently of this matching, it follows that the induced matching on the subset of \textsc{blue} half-edges that are matched to \textsc{blue} counterparts---namely, the edges retained in step (b)---is uniform.
\end{proof}

Using this, and noting
$D_i^{(n,r)} \leq D_i^{(n)}$ for all $i$, let $\sfG_n^{(r)}=([n],\sfE_n^{(r)})$ be the following random multigraph with degrees $\{ D_i^{(n,r)}\}$, coupled to the already-constructed~$\sfG_n$:
\begin{enumerate}[(a)]
\item For each $i$, mark a uniformly chosen subset of $D_i^{(n,r)}$ half-edges incident to vertex~$i$ as \textsc{blue} in $\sfG_n$.
\item Retain in $\sfG_n^{(r)}$ every edge of $\sfE_n$ where both parts are \textsc{blue}.
\item Complete the construction of $\sfG_n^{(r)}$ via a uniformly chosen matching of all unmatched half-edges.
\end{enumerate}

Let $\hat \bA_n^{(r)} = \omega_n^{-1/2} \bA_n^{(r)}$ for
$A_n^{(r)}$, the single-adjacency matrix of $\sfG_n^{(r)}$.
We next control the difference between $\cL^{\hat \bA_n}$ and $\cL^{\hat \bA_n^{(r)}}$.
Indeed, by the definition of the coupling of $\sfG_n$ and $\sfG_n^{(r)}$,  the cardinality of the symmetric $ \sfE_n\xor \sfE_n^{(r)}$ is at most twice the number of \emph{unmarked} half-edges in $\sfG_n$.
Thus,
\begin{align}  \frac1{4 n} \tr\big( (\hat \bA_n - \hat \bA_n^{(r)})^2\big) &\leq \frac1{2n\omega_n } \left|\sfE_n \xor \sfE_n^{(r)}\right|
\leq \frac{1}{n\omega_n} \sum_{i=1}^n (D_i^{(n)} - D_i^{(n,r)}) \nonumber \\
&\leq  \frac{1+o(1)}{\epsilon_r \omega_n} + \frac1r +
\frac{1}{n} \sum_{i=1}^n \hat D_i^{(n)} \one_{\{\hat D_i^{(n)} \ge r\}}
=: \eta(n,r)\,,
\label{eq:Dn-Dnr}
\end{align}
where the first term in $\eta(n,r)$  accounts for the discrepancy between
$\omega_n$ and $\omega_{n,r}$, the  term $1/r$ accounts for the degree quantization, while
the last term accounts for degree truncation
(since $d_{\ell+1}^{(r)}\geq r$). Thanks to the uniform integrability
of $\{\hat D^{(n)}_{U_n} \}$ from~\eqref{eq-Dn-moment}, we have that $\eta(r) := \limsup_{n\to\infty} \eta(n,r)$ satisfies $ \eta(r) \to 0$ as $r\to\infty$.
Furthermore,
\[ \int x^2 d\cL^{\hat \bA_n} = \frac1n \tr(\hat \bA_n^2) \leq 1+o(1)\] by the choice of $\omega_n$ in~\eqref{eq-omega-choice}, yielding the tightness of $\mu_{n,\infty} := \cL^{\hat \bA_n}$.
Altogether, we conclude from Lemma~\ref{lem:truncation} that, if $\cL^{\hat \bA_n^{(r)}}\Rightarrow \mu_r$, then $\cL^{\hat \bA_n}\Rightarrow \lim_{r\to\infty}\mu_r$.

Next, let $\omega_n^{(r)}=2 |\sfE_n^{(r)}|/n$
(as in \eqref{eq-omega-choice} but for the multigraph $\sfG_n^{(r)}$).
Since (see~\eqref{eq:Dn-Dnr}),
\[ \limsup_{n\to\infty} \bigg| 1 - \frac{\omega_n^{(r)}}{ \omega_n}\bigg| \leq \eta(r)\to 0\quad\mbox{ as $r\to\infty$}\,,\]
\abbr{wlog} we  replace $\omega_n$ by $\omega_n^{(r)}$ in the definition of $\hat \bA_n^{(r)}$, i.e., starting with
\[ \hat D_i^{(n,r)} \in \{d_1^{(r)},\ldots,d_\ell^{(r)}\} =: \cS_r\,.\]
Further, note that the hypothesis $\cL^{\hat \bLambda_n}\Rightarrow \nu_{\hat D}$ as $n\to\infty$, together with our choice of  $\cS_r$, implies that $\cL^{\hat\bLambda_n^{(r)}}$ (corresponding to $\hat\bLambda^{(r)}_n=\mathrm{diag}(\hat D^{(n,r)}_1,\ldots,\hat D^{(n,r)}_n)$)
converges weakly for each $r$ to some $\nu_{\hat D_r} \ne \delta_0$,
supported on $\RR_+$, and further,
$\nu_{\hat D_r}\Rightarrow \nu_{\hat D} \ne \delta_0$, as $r\to\infty$.

Let $\mu^{(2)}$ denote hereafter the pushforward of the measure
$\mu$ by the mapping $x \mapsto x^2$ (that is, the measure given by $B\mapsto  \mu(f^{-1}(B))$ for $f(x)=x^2$.)
 It is known that, for probability measures on $\RR_+$, the free multiplicative convolution is continuous w.r.t.\ weak convergence; that is,
$\nu_k \boxtimes \nu'_k \Rightarrow \nu \boxtimes \nu'$ provided
$\nu_k \Rightarrow \nu \ne \delta_0$, $\nu'_k \Rightarrow \nu' \ne \delta_0$
all of which are supported on $\RR_+$ (see, \emph{e.g.},~\cite[Prop. 3]{ArPe09}).
Applying this twice, we find that
\begin{equation}\label{eq:lim-sq}
\nu_{\hat D_r} \boxtimes \sigma_{\semic}^{(2)} \boxtimes \nu_{\hat D_r}
\Rightarrow \nu_{\hat D}\boxtimes \sigma_{\semic}^{(2)} \boxtimes \nu_{\hat D} \,.
\end{equation}
From this we next deduce that 
$\nu_{\hat D_r} \boxtimes \sigma_{\semic} \Rightarrow \nu_{\hat D}\boxtimes \sigma_{\semic}$. Indeed,
recall \cite[Lemma 8]{ArPe09} that the \abbr{lhs} of \eqref{eq:lim-sq}
equals $(\nu_{\hat D_r} \boxtimes \sigma_{\semic})^{(2)}$, while likewise
its \abbr{rhs} equals $(\nu_{\hat D} \boxtimes \sigma_{\semic})^{(2)}$.
For any $f \in C_b(\RR)$, the function
$g(x)=\frac{1}{2} [f(\sqrt{x})+f(-\sqrt{x})]$ is in $C_b(\RR_+)$. Thus, the weak convergence
$(\nu_{\hat D_r} \boxtimes \sigma_{\semic})^{(2)} \Rightarrow
(\nu_{\hat D} \boxtimes \sigma_{\semic})^{(2)}$, implies that 
$\nu_{\hat D_r} \boxtimes \sigma_{\semic} \Rightarrow \nu_{\hat D}\boxtimes \sigma_{\semic}$ for the corresponding 
symmetric source measures of the map $x \mapsto x^2$.
In conclusion, it suffices hereafter to prove the
theorem for the case where
$\hat D_i^{(n)}\in\cS$, a fixed finite set, for all $n$.

\medskip \noindent
{\bf Step II.}
For $1\leq a\leq \ell$, let $m_a^{(n)} = |\sfV_n^a|$ where $\sfV_n^a = \{v\in [n] : \deg(v) = d_a \omega_n \}$ is the set of vertices of degree $d_a\omega_n $ in $\sfG_n$. By assumption, $m_a^{(n)} / n \to \nu_a$ for $\nu_a := \nu_{\hat D}(\{d_a\})$.
(Observe that our choice of $\omega_n$ dictates that $\sum_a d_a\nu_a = 1$.)
For all $1\leq a,b \leq \ell$, set
\[ q_{a,b} := d_a d_b \nu_b\,.\]
Let $\sfH_n = \cup_{a\leq b} \sfH_{a,b}^{(n)}$ for the edge-disjoint multigraphs $\sfH_{a,b}^{(n)}$ that are generated by the configuration model in the following way.
\begin{itemize}
\item
For $1\leq a \leq \ell$, let $\sfH_{(a,a)}^{(n)}$ be the random $D_{a,a}^{(n)}$-regular multigraph on $\sfV_n^a$, where $D_{a,a}^{(n)} m_a^{(n)}$ is even and $\hat D_{a,a}^{(n)} := D_{a,a}^{(n)}/\omega_n$ converges to $q_{a,a}$ as $n\to\infty$.
\item For $1\leq a < b\leq \ell$, let $\sfH_{a,b}^{(n)}$ be the random bipartite multigraph with sides $(\sfV_n^a,\sfV_n^b)$ and degrees $D_{a,b}^{(n)}$ in $\sfV_n^a$ and $D_{b,a}^{(n)}$ in $\sfV_n^b$,   such that
the detailed balance
\[ D_{a,b}^{(n)}m_a^{(n)} =  D_{b,a}^{(n)}m_b^{(n)}\]
holds, and $\hat D_{a,b}^{(n)} := D_{a,b}^{(n)}/\omega_n $ tends to $ q_{a,b}$ as $n\to\infty$ (hence, $\hat D_{b,a}^{(n)}\to q_{b,a}$).
\end{itemize}
Finally, setting
\begin{equation}\label{dfn:lamb-ab}
\lambda_{a,b}^{(n)} = \frac{\omega_n }n d_a d_b \,,
\end{equation}
let $\widetilde \bA_n$ denote the single-adjacency matrix of the multigraph
$\widetilde \sfH_n = \cup_{a \leq b} \widetilde \sfH_{a,b}^{(n)}$, where the edge-disjoint multigraphs $\widetilde \sfH_{a,b}^{(n)}$ are defined as follows.
\begin{itemize}
\item For $1\leq a \leq b \leq \ell$, mutually independently set the multiplicity of the edge between distinct $i\in \sfV_a^n$ and $j\in \sfV_b^n$ in $\widetilde \sfH_{a,b}^{(n)}$ to be a  $\PPo(\lambda_{a,b}^{(n)})$ random variable.
\item For $1\leq a \leq \ell$, mutually independently set the number of loops incident to $i\in \sfV_n^a$ to be a $ \PPo(\frac12\lambda_{a,a}^{(n)})$ random variable.
\end{itemize}
Our next proposition shows that
$\cL^{\hat \bA_n}\Rightarrow \nu_{\hat D}\boxtimes \sigma_{\semic}$, in probability,
whenever
\begin{equation}\label{dfn:ind-adj}
\cL^{\omega_n^{-1/2} \widetilde \bA_n} \Rightarrow  \nu_{\hat D}\boxtimes \sigma_{\semic}
\,, \quad \hbox{ in probability}\,.
\end{equation}

\begin{proposition}\label{prp:coup}
The empirical spectral measures of  $\bA_n,  \bA'_n$ and $  \widetilde \bA_n$, the respective single-adjacency matrices of $\sfG_n, \sfH_n$ and $\widetilde \sfH_n$, satisfy
\[  d_{\textsc{bl}}\left(\cL^{\omega_n^{-1/2} \bA_n}, \cL^{\omega_n^{-1/2} \bA'_n}\right) =o(1)\quad\mbox{ and }\quad
 d_{\textsc{bl}}\left(\cL^{\omega_n^{-1/2} \bA'_n}, \cL^{\omega_n^{-1/2} \widetilde \bA_n}\right) = o(1)\,, \]
 in probability, as $n\to\infty$.
\end{proposition}
\begin{proof}
Setting
\[ \sfG_n^{(0)} = \sfG_n\,,\quad \sfG_n^{(2)} = \sfH_n\,,\quad \sfG_n^{(4)} = \widetilde \sfH_n\,,\]
associate with each multigraph its sub-degrees
(accounting for edge multiplicities),
\[
D_{i,b}^{(n,k)} := \sum_{j\in \sfV_n^b} (\bA_{\sfG_n^{(k)}})_{i,j}\,,
\qquad i\in [n]\,, \qquad 1 \leq b\leq \ell \,,
\]
so in particular
$D_{i,b}^{(n,2)} = D_{a(i),b}^{(n)}$ where
$a(i)$ is such that $i\in \sfV_n^a$. Of course, for $k=0,2,4$,
\begin{equation}\label{e:dbal}
m^{(n,k)}_{a,b} := \sum_{i \in \sfV_n^a} D_{i,b}^{(n,k)} = m^{(n,k)}_{b,a}
\,, \quad m^{(n,k)}_{a,a} \; \mbox{ is even, }
\qquad
1 \le a, b \le \ell \,.
\end{equation}

\begin{claim}\label{clm:multigraphs-esd}
Conditional on a given sequence of sub-degrees $\{D_{i,b}^{(n,k)}\}$, the adjacency matrices $\bA_{\sfG_n^{(k)}}$ for $k\in\{0,2,4\}$ all have the same conditional law.
\end{claim}
\begin{proof} Observe that $\sfG_n=\sfG_n^{(0)}$ gives the same weight to each perfect matching of its half-edges, thus conditioning on $\{D_{i,b}^{(n,k)}\}$ amounts to specifying a subset of permissible matchings, on which the conditional distribution would be uniform.
The same applies to the graphs $\sfH_{(a,b)}^{(n)}$ for all $1\leq a\leq b\leq \ell$, each being an independently drawn uniform multigraph, and hence to their union $\sfH_n=\sfG_n^{(2)}$, thus establishing the claim for $k=0,2$. To treat $k=4$, notice that the probability that the multigraph $\sfH_{(a,b)}^{(n)}$, $a\neq b$, given the sub-degrees $\{D_{i,b}^{(n,k)}\}$, features the adjacency matrix $\ba:=(a_{i,j})$ ($i \in  \sfV_n^a$, $j\in \sfV_n^b$), is
\[ \frac1{m_{a,b}^{(n,k)}!}\bigg(\prod_{i\in \sfV_n^a}  \frac{D_{i,b}^{(n,k)}!}{\prod_{j\in \sfV_n^b} a_{i,j}!}\bigg)\bigg(\prod_{j\in \sfV_n^b} D_{j,a}^{(n,k)} !\bigg) \propto \prod_{i \in \sfV_n^a}\prod_{j \in \sfV_n^b} \frac{1}{a_{i,j}!}
\]
by the definition of the configuration model. As the distribution of a vector of $t$ i.i.d.\ Poisson variables with mean $\lambda$, conditional on their sum being $m$, is multinomial with parameters $(m,\frac1t,\ldots,\frac1t)$, the analogous conditional probability under $\widetilde \sfH_{(a,b)}^{(n)}$ is
\[
 \prod_{i\in \sfV_n^a} \frac{D_{i,b}^{(n,k)}! }{\prod_{j\in \sfV_n^b} a_{i,j}!} |\sfV_n^b|^{-D_{i,b}^{(n,k)}} \propto \prod_{i \in V^n_a} \prod_{j \in \sfV_n^b} \frac1{a_{i,j}!}\,.
\]
Lastly, the probability that $\sfH_{(a,a)}^{(n)}$, conditional on  $\{D_{i,b}^{(n,k)}\}$, assigns to $\ba=(a_{i,j})$ is
\[
\prod_{i\in  \sfV_n^a} \frac{D_i !}{2^{a_{i,i}}} \prod_{\substack{j\in \sfV_n^a \\ j>i}} \frac1{a_{i,j}!} \propto 2^{-\sum_i a_{i,i}}
\prod_{\substack{i,j\in \sfV_n^a \\ j>i}}
\frac1{a_{i,j}!}\,,
\]
whereas the analogous conditional probability under $\widetilde{\sfH}_{(a,b)}^{(n)}$ (now involving a vector that is multinomial with parameters $(D_{i,b}^{(n,k)},\frac{1}{2t+1},\frac{2}{2t+1},\ldots,\frac{2}{2t+1})$ for $t=|\{ j \in \sfV_n^a :  j\geq i\}|$, recalling the factor of 2 in the definition of the rate of loops under $\widetilde H^{(n)}_{(a,a)}$), is
\[
 \prod_{i\in  \sfV_n^a} \frac{D_{i,b}^{(n,k)}! }{\prod_{
\substack{j \in \sfV_n^a \\  j>i}} a_{i,j}!} 2^{-a_{i,i}} \bigg(\frac2{|\{j\in \sfV_n^a:  j\geq i\}|}\bigg)^{-D_{i,b}^{(n,k)}}  \propto 2^{-\sum_i a_{i,i}}
\prod_{\substack{i,j\in \sfV_n^a \\ j>i}}
\frac1{a_{i,j}!}\,.
\]
This completes the proof of the claim.
\end{proof}

We will introduce two auxiliary
multigraphs $\sfG_n^{(1)}$ and $\sfG_n^{(3)}$
having the latter property, and further, the
corresponding single-adjacency matrices
(or single-edge sets $\sfE_n^{(k)}$), can be coupled in such a way that
\begin{equation}\label{e:k-coup}
\sum_{k=1}^4 \EE \Big[ \big| \sfE_n^{(k)} \xor \sfE_n^{(k-1)}\big| \Big]
= o(n \omega_n)\,.
\end{equation}
It follows that, under the resulting coupling, both
$\EE[\tr\big( (\bA_n -\bA'_n)^2\big)] = o(n \omega_n)$ and
$\EE[\tr\big( (\bA'_n -\widetilde \bA_n)^2\big)] = o(n \omega_n)$,
yielding Proposition~\ref{prp:coup} via the Hoffman--Wielandt bound.

Proceeding to construct the multigraph $\sfG_n^{(1)} $, write, for all $i\in [n]$ and $1 \leq b\leq \ell$,
\begin{align}\label{e:pre-bal}
D_{i,b}^{(n,1)} &= D_{i,b}^{(n,0)} \wedge D_{i,b}^{(n,2)}\,,
\end{align}
then further uniformly reduce the number
of potential half-edges in $\sfG_n^{(1)}$ until achieving \eqref{e:dbal} for $k=1$.
That is, if \eqref{e:pre-bal} yields $m_{a,b}^{(n,1)} > m_{b,a}^{(n,1)}$
for some $a \ne b$, we uniformly choose and eliminate
$m_{a,b}^{(n,1)}-m_{b,a}^{(n,1)}$ potential half-edges leading
from $\sfV_n^a$ to $\sfV_n^b$ and accordingly
adjust $\{D^{(n,1)}_{i,b}, i \in \sfV_n^a\}$, an operation which only affects the constraint \eqref{e:dbal} for that particular $a \ne b$.
With Observation~\ref{obs:coupling} in mind,  construct two \emph{bridge} copies of the random multigraph
$\sfG_n^{(1)}$ with the adjusted sub-degrees $\{D^{(n,1)}_{i,b}\}$, as follows:
\begin{itemize}
\item For each $i$ and $b$, mark as \textsc{blue($b$)} a uniformly chosen
subset of $D_{i,b}^{(n,1)}$ half-edges incident to vertex $i$, the other
part of which is, according to $\sfG_n^{(0)}$, in $\sfV_n^b$.
\item Retain for $\sfG_n^{(1)}$ every edge of $\sfG_n^{(0)}$ where both parts
are marked with \textsc{blue}.
\item After removing all non-\textsc{blue} half-edges of $\sfG_n^{(0)}$, complete
the construction of $\sfG_n^{(1)}$ by uniformly matching, for each $a \ge b$,
all unmatched \textsc{blue$(b)$} half-edges of $\sfV_n^a$
to all unmatched \textsc{blue$(a)$} half-edges of $\sfV_n^b$.
\item A second copy of $\sfG_n^{(1)}$ is obtained by repeating the preceding
construction, now with $\sfG_n^{(2)}$ taking the role of $\sfG_n^{(0)}$.
\end{itemize}
Replacing in the above procedure the multigraph $\sfG_n^{(0)}$ by the multigraph
$\sfG_n^{(4)}$, the same construction produces a multigraph $\sfG_n^{(3)}$
having sub-degrees
\begin{equation}\label{e:bridge}
D^{(n,3)}_{i,b} \le D^{(n,2)}_{i,b} \wedge D^{(n,4)}_{i,b} \,,
\end{equation}
and two \emph{bridge} copies of $\sfG_n^{(3)}$ which are
coupled (using such \textsc{blue} marking), to $\sfG_n^{(2)}$ and $\sfG_n^{(4)}$,
respectively.

Next, as for \eqref{e:k-coup}, recall that
$|\sfE_n^{(k)} \xor \sfE_{n}^{(k-1)}| \le |\sfE_{\sfG_n^{(k)}} \xor \sfE_{\sfG_n^{(k-1)}}|$,
which under our coupling is at most the number of edges of
$\sfG_n^{(2 [k/2])}$ that had at least one non-\textsc{blue} part.
This in turn is at most
\[
\Delta^{(n)} := \sum_{a,b=1}^{\ell} |m^{(n,k)}_{a,b} - m^{(n,k-1)}_{a,b}| \,.
\]
Our construction is such that
$m^{(n,0)}_{a,b} \wedge m^{(n,2)}_{a,b} \ge m^{(n,1)}_{a,b}$
and
$m^{(n,4)}_{a,b} \wedge m^{(n,2)}_{a,b} \ge m^{(n,3)}_{a,b}$.
Further, if the sub-degrees of bridge multigraphs were set by
\eqref{e:pre-bal}, then
\[
m^{(n,0)}_{a,b} + m^{(n,2)}_{a,b} - 2 m^{(n,1)}_{a,b}
= \sum_{i \in \sfV_n^a} |D^{(n,0)}_{i,b} - D^{(n)}_{a,b}| :=
\Delta^{(n,1)}_{a,b} \,,
\]
for any $1 \le a,b \le \ell$, with analogous identities
relating $m^{(n,3)}_{a,b}$ and $\Delta^{(n,3)}_{a,b}$.
Since \eqref{e:dbal} holds for $k=0,2,4$, while
$m^{(n,1)}_{a,b} \wedge m^{(n,1)}_{b,a}$, $b < a$ are
not changed by the $\sfG_n^{(1)}$ sub-degree adjustments
(and similarly for the $\sfG_n^{(3)}$ sub-degree adjustments),
we deduce that
\[
\Delta^{(n)} \le 2 \sum_{a,b=1}^{\ell} \Delta^{(n,1)}_{a,b}
+ 2 \sum_{a,b=1}^{\ell} \Delta^{(n,3)}_{a,b} \,.
\]
Thus, we have \eqref{e:k-coup} as soon as we show that for any
$1 \le a,b \le \ell$,
\[
\EE \Delta^{(n,1)}_{a,b} + \EE \Delta^{(n,3)}_{a,b} = o(n \omega_n) \,,
\]
which by our choice of $\{D^{(n)}_{a,b}\}$ follows from having for
any fixed $i \in \sfV_n^a$,
\begin{equation}\label{e:final-est}
\EE |\omega_n^{-1} D^{(n,0)}_{i,b} - q_{a,b}|
+\EE |\omega_n^{-1} D^{(n,4)}_{i,b} - q_{a,b}| = o(1) \,.
\end{equation}
For $i \in \sfV_n^a$ the variable $D^{(n,4)}_{i,b}$ is Poisson with mean $(1+o(1))\lambda_{a,b}^{(n)} m_b^{(n)} = \omega_n q_{a,b} (1+o(1))$
(see \eqref{dfn:lamb-ab}), hence
$\EE |\omega_n^{-1} D^{(n,4)}_{i,b} - q_{a,b}| \to 0$. Similarly,
$D^{(n,0)}_{i,b}$ counts how many of the $d_a \omega_n$
half-edges emanating from such $i$, are paired by the uniform matching
of the
half-edges of $\sfG_n$,
with half-edges from the subset $\sfE_n^b$ of those incident to $\sfV_n^b$.
With $|\sfE_n^b| = d_b \omega_n m_b^{(n)}$, the probability of a specific
half-edge paired with an element of $\sfE_n^b$ is
$\mu_n = (|\sfE_n^b|-1_{\{a = b\}})/(2 |\sfE_{\sfG_n}|-1)  \to d_b \nu_b$, hence
$\omega_n^{-1} \EE D^{(n,0)}_{i,b} = d_a \mu_n \to q_{a,b}$.
It is not hard to verify that two specific half-edges
incident to $i \in \sfV_n^a$ are both paired with elements of $\sfE_n^b$
with probability $v_n = \mu_n^2 (1+o(1))$. Consequently,
\[
{\rm Var} (\omega_n^{-1} D^{(n,0)}_{i,b}) \le d_a \frac{\mu_n}{\omega_n}
+ d_a^2 (v_n - \mu_n^2) \to 0\,,
\]
yielding
the $L^2$-convergence of
$\omega_n^{-1} D^{(n,0)}_{i,b}$ to $q_{a,b}$ and thereby establishing
\eqref{e:final-est}.
\end{proof}

\medskip\noindent
{\bf Step III.} We proceed to verify \eqref{dfn:ind-adj} for the single-adjacency
matrices $\widetilde \bA_n$ of $\widetilde \sfH_n$. To this end, as argued before,
such weak convergence as in \eqref{dfn:ind-adj} is not affected by changing $o(n \omega_n)$
of the entries of $\widetilde \bA_n$, so \abbr{wlog} we modify the
law of number of loops in $\widetilde \sfH_n$ incident to each $i \in \sfV_n^a$ to
be a $\PPo(\lambda_{a,a}^{(n)})$ variable, yielding the
symmetric matrix $\widetilde \bA_n$ of independent upper triangular
Bernoulli($p^{(n)}_{a,b}$) entries, where
$p^{(n)}_{a,b}=1-\exp(-\lambda^{(n)}_{a,b})
$
when $i \in \sfV_n^a$ and $j \in \sfV_n^b$. In particular, the rank of
$\EE \widetilde \bA_n$ is at most $\ell$, so by Lidskii's theorem we get
\eqref{dfn:ind-adj} upon proving that
$\cL^{\hat \bB_n} \Rightarrow
\nu_{\hat D}\boxtimes \sigma_{\semic}$ in probability, for
$\hat \bB_n := \omega_n^{-1/2} (\widetilde \bA_n - \EE \widetilde \bA_n)$, a
symmetric matrix of uniformly (in $n$) bounded, independent upper-triangular entries $\{\hat
Z_{ij}\}$,
having zero mean and variance
$v^{(n)}_{a,b} := \omega_n^{-1} p^{(n)}_{a,b} (1-p^{(n)}_{a,b})
= \frac{1}{n} d_a d_b (1+o(1))$ when $i \in \sfV_n^a$, $j \in \sfV_n^b$.
As a special case of Remark~\ref{rem:goe} (corresponding to piecewise-constant diagonal matrices with values $\{d_a\}_{a=1}^\ell$), 
such convergence holds for the symmetric matrices $\bB_n$, whose
independent centered Gaussian entries $Z_{ij}$ have
variance $v^{(n)}_{a,b}$ when $i \in \sfV_n^a$ and $j \in \sfV_n^b$, subject
to on-diagonal rescaling $\EE Z_{ii}^2 = 2 v^{(n)}_{a(i),a(i)}$.
As in the classical proof of Wigner's theorem by the moment's method
(cf.~\cite[Sec. 2.1.3]{AGZ10}), it is easy to check
that for any fixed $k=1,2,\ldots$,
\[
\EE \Big[\frac{1}{n}\tr(\hat \bB_n^k) \Big] = \EE \Big[
\frac{1}{n}\tr (\bB_n^k) \Big] (1+o(1)) \,,
\]
since both expressions are dominated by those cycles of length $k$
that pass via each entry of the relevant matrix exactly twice
(or not at all). Further, adapting the concentration argument of
\cite[Sec.~2.1.4]{AGZ10} we deduce that as in the Wigner's case,
$\langle x^k, \cL^
{\hat \bB_n} - \EE\cL^{ \hat \bB_n} \rangle \to 0$ in probability,
for each fixed $k$,
thereby completing the proof of Theorem~\ref{thm:1}(a).

To prove Theorem~\ref{thm:1}(b), recall that  $|\sfE_n\triangle \sE_n| \le \bDelta_n$ for any coupling of the 
pair of matching which generate the graphs  $\sfG_n$ and $\sG_n$. 
Appealing to Proposition~\ref{prop:coupling} and the bound~\eqref{eq:deltaEn-simple}  following it,
we get that under the coupling $\mu$ provided by that proposition,
\[
    \EE_\mu [ |\sfE_n\triangle \sE_n| ] \le \EE_\mu [ \bDelta_n ] \le \sqrt{2 \EE_\mu [ \bDelta_n (\bDelta_n-1) ]} \le
    4 b_n \,,
\] 
where (recalling from~\eqref{eq-omega-choice} that $\omega_n=(2+o(1))|\sfE_n|/n$)
\begin{align*}  b_n^2 :=  \sqrt{2|\sfE_n|}  \sum_{j=1}^n D_j^2  &= (1+o(1)) n^{3/2}\sqrt{\omega_n}  \EE_{U_n} (D^{(n)}_{U_n})^2 \\
&= (1+o(1)) n^{3/2}\omega_n^{5/2}  \EE_{U_n} (\hat{D}^{(n)}_{U_n})^2
 = o(n^2 \omega_n^2)\end{align*}
via our assumption on the \abbr{rhs} of~\eqref{eq-Dn-moment}; thus, $ \EE_\mu [| \sfE_n \xor \sE_n|] = o(n\omega_n)$.
We claim that Lemma~\ref{lem:truncation} then concludes the proof. To see this, set 
$\hat{\bB}'_{n} \equiv \omega_n^{-1/2} \bA_{\sG_n}$ and further let
$\hat{\bA}_n \equiv \omega_n^{-1/2}\bA_n$ for the single-adjacency matrix $\bA_n$ associated with
$\bA_{\sfG_n}$. Since the entries of $\bA_n$ and $\bA_{\sG_n}$ may differ at most by one from each other, \eqref{eq-assumption-En} implies that
\[
    \EE_\mu \Big[\frac1n \tr\big((\hat{\bA}_n - \hat{\bB}'_n)^2\big) \Big]
    \le
    \frac{2}{n\omega_n} 
    \EE_\mu [ |\sfE_n\xor\sE_n| ]
    \to 0\,,
\]
as required for Lemma \ref{lem:truncation}.
\end{proof}

\begin{proof}[\textbf{\emph{Proof of Corollary~\ref{cor-iid}}}] 
The assumed growth of $\omega_n$ yields \eqref{eq-assumption-En}
out of \eqref{eq-omega-choice}. In case of $\sfG_n$, the latter amounts to
\begin{equation}\label{e:L2-wlln}
\frac1n \sum_{i=1}^n \hat D_i^{(n)} \to 1\,, \qquad \mbox{ in probability},
\end{equation}
which we get by applying the $L^2$-\abbr{wlln} for triangular arrays
with uniformly bounded second moments. The same reasoning
yields the required uniform integrability in \eqref{eq-Dn-moment}, namely,
that when $n \to \infty$ followed by $r \to \infty$
\begin{equation}\label{e:UI}
\frac1n \sum_{i=1}^n \hat D_i^{(n)} 1_{\{\hat D^{(n)}_i \ge r\}} \to 0\,, \qquad \mbox{ in probability}.
\end{equation}
Further, applying the weak law for non-negative triangular arrays
$\{(\hat D_i^{(n)})^2\}$ of uniformly bounded mean, at truncation level $b_n \gg n$, 
it is not hard to deduce that
\begin{equation}\label{e:weak-lln-d2}
\frac{1}{b_n} \sum_{i=1}^n (\hat D_i^{(n)})^2 \to 0\,, \qquad \mbox{  in probability},
\end{equation}
whereupon, considering
$b_n = n/\sqrt{\omega_n/n}$ results with  the \abbr{rhs} of \eqref{eq-Dn-moment}. Next, recall
that the empirical measures $\cL^{\hat \bLambda_n}$ of
i.i.d.\ $\hat D_i^{(n)}$ converge in probability
to the weak limit $\nu_{\hat D}$ of the laws of $\hat D_1^{(n)}$. Thus, 
Theorem~\ref{thm:1}(a) applies for $\sfG_n$ of 
degrees $[\omega_n \hat D_i^{(n)}]$,
yielding Corollary~\ref{cor-iid} in this case. 

Turning to the case of uniform simple graphs, thanks to \eqref{e:UI}, 
truncating the degrees $[\omega_n \hat D_i^{(n)}]$ at some $\bar d_n \gg \omega_n$
removes at most $o(n \omega_n)$ edges from $\sfE_n$. Thus, such truncation neither 
affects \eqref{eq-omega-choice}, nor the preceding  verification of \eqref{eq-Dn-moment}.
Further, such truncation alters only $o(n)$ degrees, yielding the same limit $\nu_{\hat D}$
for $\cL^{\hat \bLambda_n}$. In view of
Theorem~\ref{thm:1}(b), the stated convergence of $\cL^{\hat \bA_{\sG_n}}
$ holds, 
provided that $\{ [\omega_n \hat D_i^{(n)}] \wedge \bar d_n \}$ are graphical \abbr{whp}
as $n \to \infty$. To this end, inspired by the proof of 
\cite[Theorem 1(d)]{AL05}, recall from the Erd\H{o}s-Gallai theorem, 
that integers $d_1 \ge d_2 \ge \cdots \ge d_n \ge 0$ are graphical if 
\begin{equation}\label{eq:Er-Ga}
2 \sum_{i=1}^j d_i \le j(j-1) + \sum_{i=1}^n \min(j,d_i) \,, \qquad \forall \; 1 \le j \le n \,.
\end{equation} 
Thanks to \eqref{eq-assumption-En} we can fix 
$j_n=o(n)$ 
such that $j_n / \sqrt{n \omega_n} \to \infty$. The \abbr{lhs} of \eqref{eq:Er-Ga}
is in our setting at most $2 \omega_n \sum_i \hat{D}_i^{(n)}$, which in view of 
\eqref{e:L2-wlln} is for $j>j_n$ negligible in comparison with 
the term $j(j-1)$ on the \abbr{rhs} of \eqref{eq:Er-Ga}. Denoting by 
$o_p(1)$ the \abbr{lhs} of \eqref{e:weak-lln-d2} at $b_n = n^2/j_n \gg n$,
we further have here that the \abbr{lhs} of \eqref{eq:Er-Ga} is at most 
\begin{equation}\label{eq:lhs}
2 \min \Big( j \bar d_n, n \omega_n \, o_p(1) \Big) \,, \qquad \forall \; 1 \le j \le j_n \,.
\end{equation}
The Paley--Zygmund inequality yields $\inf_n \PP(\hat D_i^{(n)} \ge 2/3) \ge 2 \delta$, for  some 
$\delta > 0$. Hence,
\[
\liminf_{n \to \infty} \frac{1}{n} \sum_{i=1}^n \one_{\{D_i^{(n)} \ge \omega_n/3\}} > \delta\,, \qquad \mbox{ in probability}. 
\]
This yields that the right-most term in  \eqref{eq:Er-Ga} is for all large $n$ and $j \in [n]$, at least 
\[
\delta \min(j n, n \omega_n/3) \,,
\] 
which in turn exceeds 
\eqref{eq:lhs} (as $\bar d_n = o(n)$), thus completing the proof.
\end{proof}

\section{Coupling simple graphs and multigraphs: Proof of Proposition \ref{prop:coupling}}\label{sec:coupling}

Fixing \emph{graphical} degrees $D_1 \ge D_2 \ge \cdots \ge D_n$, let 
$\sfm_n := \sum_{i} D_i = 2|\sfE_n|$. Enumerate 
the $\sfm_n$ half-edges as follows: each half-edge $e$ is identified with a vertex $v(e) \in [n]$; the first $D_1$ half-edges have $v(e)=1$, the next $D_2$ 
have $v(e)=2$ and so on. A matching of half-edges $m: [\sfm_n] \mapsto [\sfm_n]$ is
an involution without fixed points (i.e., $m(e)=m^{-1}(e)$ and
$m(e) \ne e$ for all $e \in [\sfm_n]$). A coupled pair of multigraphs $(\sfG_n,\sG_n)$ 
is hereby represented by a pair of matching $(\sfX,\sfY)$,
restricting $\sfY(\cdot)$ to the non-empty collection of matching that 
correspond to a simple graph; namely, $v(e) \ne v(\sfY(e))$ (no loops)
and $\{v(e),v(\sfY(e))\} \ne \{v(f),v(\sfY(f))\}$ (no multiple edges) for any $f \ne \{e,\sfY(e)\}$. 

\begin{figure}
\centering
  \begin{tikzpicture}  
  \foreach \a / \b in {0/0,0/1,1/0,1/1}
  {
  \begin{scope}[shift={(\a * 7.5, -4.5 * \b)}]  
  \path [rounded corners,fill=black!3] (0.6,1) rectangle (5.6,-3);
	\foreach \y in {1,...,4}
	{\begin{scope}[shift={(\y * 1.25, 0)}]
            \node [circle,fill=black!50,scale=0.6] (VT\a\b\y) at (0,0) {};
            	\node (ET\a\b\y) at (0,-.66) {};
            	\draw [color=black](VT\a\b\y) -- (ET\a\b\y) ;
            \draw [dashed] (0, -0.18) circle [x radius=0.2, y radius=0.35];
    \end{scope}
    }
	\foreach \y in {1,...,2}
    	{\begin{scope}[shift={( 1.25 * \y + 1.25, -2.5)}]
            \node [circle,fill=black!50,scale=0.6] (VB\a\b\y) at (0,0) {};
            	\node (EB\a\b\y) at (0,.66) {};
            	\draw [color=black](VB\a\b\y) -- (EB\a\b\y) ;
            \draw [dashed] (0, 0.18) circle [x radius=0.2, y radius=0.35];
    \end{scope}
  	}
  	\node [label={\small $e$}, xshift=-13pt, yshift=-4pt] at (VB\a\b1) {};
    \node [label={\small $f$}, xshift=13pt, yshift=-4pt] at (VB\a\b2) {};
  	\end{scope}
  }
  \foreach \b in {0,1} {
    \node [label={\small $\sfX_0(e)$}, yshift=4pt] at (VT0\b1) {};
    \node [label={\small$\sfX_0(f)$}, yshift=4pt] at (VT0\b2) {};
    \node [label={\small $\sfY_0(e)$}, yshift=4pt] at (VT1\b3) {};
    \node [label={\small$\sfY_0(f)$}, yshift=4pt] at (VT1\b4) {};
  }
  \path[very thick,purple] (EB001.south) edge [bend left=40] (ET001.north);
  \path[very thick,purple] (EB002.south) edge [bend left=30] (ET002.north);
  \path[very thick,blue] (EB101.south) edge [bend right=30] (ET103.north);
  \path[very thick,blue] (EB102.south) edge [bend right=40] (ET104.north);
  \path[very thick,purple] (EB011.south) edge [bend left=50] (EB012.south);
  \path[very thick,purple] (ET011.north) edge [bend right=50] (ET012.north);
  \path[very thick,blue] (EB111.south) edge [bend left=50] (EB112.south);
  \path[very thick,blue] (ET113.north) edge [bend right=50] (ET114.north);
  \node at (6.85,-1.1) {\large $(\sfX_0,\sfY_0)$};
  \node at (6.85,-5.6) {\large $(\sfX_1,\sfY_1)$};
\end{tikzpicture}
  \vspace{-0.3cm}
 \caption{Coupling of the chains $(\sfX_t,\sfY_t)$ corresponding to $(\sfG_n,\sG_n)$.}
  \label{fig:coupling-multi-simple}
  \vspace{-0.2cm}
\end{figure}

Starting from any such pair of matching $(\sfX_0,\sfY_0)$, consider the switching
Markov chain $(\sfX_k,\sfY_k)$ that proceeds as following (see also Figure~\ref{fig:coupling-multi-simple}):
\begin{itemize}
\item Uniformly choose $e \ne f \in [\sfm_n]$ and disconnect their matching in $\sfX_k$ and $\sfY_k$;
\item Reconnect $e$ with $f$, and $\sfX_k(e)$ with $\sfX_k(f)$, to get the match $\sfX_{k+1}$;
\item If reconnecting $e$ with $f$ and $\sfY_k(e)$ with $\sfY_k(f)$ yields a simple graph, set
this to be $\sfY_{k+1}$. Otherwise, leave $\sfY_{k+1} = \sfY_k$ unchanged.
\end{itemize}
We say that coupling succeeds in the $k$-th step if the proposed move 
to $\sfY_{k+1}$ results in a simple graph, otherwise saying that the coupling failed
(in the $k$-th step). 

The marginal $(\sfX_k)$ evolves as a Markov chain in the space of all matching, 
with the marginal $(\sfY_k)$ likewise evolving as 
a Markov chain in the non-empty subset of all matching that correspond 
to simple graphs with the specified degrees. These switching chains are
further reversible with respect to the corresponding uniform measures.  Both marginal chains 
have been extensively studied as means of sampling uniform graphs subject to given degrees.
In particular, it is well-known (\cite{Taylor81}; cf.\ also the recent work~\cite{GS18})
that each of these marginals is an irreducible Markov chain. 
Having a non-empty finite state space, the Markov chain $(\sfX_k,\sfY_k)$ 
admits an invariant probability measure $\mu$, and by the preceding, any such $\mu$ is 
a coupling between the random multigraph $\sfG_n$ and the corresponding 
uniformly simple graph $\sG_n$ of the specified degrees. 

Denoting by 
\[
    \sfC_k \equiv \{e\in [\sfm_n]: \sfX_k(e)=\sfY_k(e)\} \,,
\]
the common part of the two matching $\sfX_k,\sfY_k$,  note that under an invariance measure
$\EE_\mu [|\sfC_k|]$ must be independent of $k$. We further have the following lower bound
on the change between $|\sfC_{k+1}|$ and $|\sfC_k|$:
\begin{equation}\label{clm:ob1}
|\sfC_{k+1}| - |\sfC_k| \ge 2 \one_{\{e,f \notin \sfC_k\}} - 4 \one_{\{\mbox{coupling fails in step $k$}\}} \,.
\end{equation}
Indeed, \eqref{clm:ob1} is verified by enumerating over the seven possible 
cases for $e,f\in [\sfm_n]$: 
\begin{enumerate}[\quad I.]
    \item $\sfX_0(e) = \sfY_0(e) = f $;  
    \item $\sfX_0(e) = \sfY_0(e)\neq f$, $\sfX_0(f) = \sfY_0(f)$;
  \item $\sfX_0(e) = \sfY_0(e)\neq f$, $\sfX_0(f) \neq \sfY_0(f)$ 
            or $\sfX_0(f) = \sfY_0(f)\neq e$, 
            $\sfX_0(e) \neq \sfY_0(e)$;  
  \item $\sfX_0(e) = f\neq \sfY_0(e)$ 
            or $\sfY_0(e) = f\neq \sfX_0(e)$;
    \item $\sfX_0(e) = \sfY_0(f)$, $\sfX_0(f) = \sfY_0(e)$;
    \item $\sfX_0(e) = \sfY_0(f)$, $\sfX_0(f) \neq \sfY_0(e)$
            or $\sfX_0(f) = \sfY_0(e)$, $\sfX_0(e) \neq \sfY_0(f)$;
    \item $e,f, \sfX_0(e),\sfX_0(f),\sfY_0(e),\sfY_0(f)$
        are six distinct half-edges.
\end{enumerate}
{
\tabcolsep=0.10cm
\begin{table}[h]
\hspace{-0.15cm}
\begin{tabular}{clcc}
\toprule%
Case & \qquad\qquad\qquad Criterion & \multicolumn{2}{c}{\qquad $|\sfC_1|-|\sfC_0|$}  \\
  &  & success &  failure     \\
\midrule
I & 
{\footnotesize $\sfX_0(e) = f = \sfY_0(e) $}
\quad\; \raisebox{-6pt}{
  \begin{tikzpicture}
     \node[circle,scale=0.4,fill=black,label={\tiny$e$}] (e) at (0,0) {};
    \node[circle,scale=0.4,fill=black,label={\tiny$f$}] (f) at (1.25,0) {};
       \path[thick,purple] (e) edge [bend left=30] (f);
        \path[dashed, thick,blue] (e) edge [bend right=30] (f);
    \end{tikzpicture}
}
& 0
& --- \\
\noalign{\smallskip}
\midrule[0.25pt]
II&
{\footnotesize\!\!\begin{tabular}{l} $\sfX_0(e) = \sfY_0(e)\neq f$ \\
$\sfX_0(f) = \sfY_0(f)$\end{tabular}}
\quad\;
\raisebox{-12pt}{ \begin{tikzpicture}
    \node[circle,scale=0.4,fill=black,label={\tiny$e$}] (e) at (0,0) {};
    \node[circle,scale=0.4,fill=black,label={\tiny$f$}] (f) at (1.25,0) {};
        \node[circle,scale=0.4,fill=black] (u) at (0,-0.75) {};
	\node[circle,scale=0.4,fill=black] (v) at (1.25,-0.75) {};
        \path[thick,dashed,blue] (e) edge [bend left=30] (u);
        \path[thick,purple] (e) edge (u);
	\path[thick,purple] (f) edge  (v);
	\path[thick,dashed,blue] (f) edge [bend left=30] (v);
    \end{tikzpicture}}
&
$0$
&
$-4$ \\
\noalign{\smallskip}
\midrule[0.25pt]
\raisebox{-3pt}{III}&
{\footnotesize
\!\! \!\!\!\!\begin{tabular}{l}
  \begin{tabular}{l}$\sfX_0(e) = \sfY_0(e)\neq f$ \\
 $\sfX_0(f) \neq \sfY_0(f)$
 \end{tabular}
\qquad\;\;\raisebox{-12pt}{\begin{tikzpicture}
    \node[circle,scale=0.4,fill=black,label={\tiny$e$}] (e) at (0.375,0) {};
    \node[circle,scale=0.4,fill=black,label={\tiny$f$}] (f) at (1.125,0) {};
        \node[circle,scale=0.4,fill=black] (u) at (0,-0.75) {};
	\node[circle,scale=0.4,fill=black] (v) at (.75,-0.75) {};
	\node[circle,scale=0.4,fill=black] (w) at (1.5,-0.75) {};
        \path[thick,dashed,blue] (e) edge [bend right=30] (u);
        \path[thick,purple] (e) edge (u);
	\path[thick,purple] (f) edge  (v);
	\path[dashed, thick,blue] (f) edge  (w);
    \end{tikzpicture}} 
 \\ 
 \noalign{\medskip}
 \hline
  \begin{tabular}{l}   $\sfX_0(f) = \sfY_0(f)\neq e$ \\
            $\sfX_0(e) \neq \sfY_0(e)$\end{tabular}
\qquad\;  \raisebox{-12pt}{\begin{tikzpicture}
    \node[circle,scale=0.4,fill=black,label={\tiny$e$}] (e) at (0.3755,0) {};
    \node[circle,scale=0.4,fill=black,label={\tiny$f$}] (f) at (1.125,0) {};
        \node[circle,scale=0.4,fill=black] (u) at (0.75,-0.75) {};
	\node[circle,scale=0.4,fill=black] (v) at (1.5,-0.75) {};
	\node[circle,scale=0.4,fill=black] (w) at (0,-0.75) {};
        \path[thick,dashed,blue] (f) edge [bend left=30] (v);
        \path[thick,purple] (f) edge (v);
	\path[thick,purple] (e) edge  (u);
	\path[dashed, thick,blue] (e) edge  (w);
       \end{tikzpicture}} 
            \end{tabular}}
& $0$ 
& $-2$
\\
\noalign{\smallskip}
\midrule[0.25pt]
\raisebox{-3pt}{IV}&
{\footnotesize \!\!\begin{tabular}{l}$\sfX_0(e) = f\neq \sfY_0(e)$  \qquad\;
 \raisebox{-10pt}{\begin{tikzpicture}
    \node[circle,scale=0.4,fill=black,label={\tiny$e$}] (e) at (0,0) {};
    \node[circle,scale=0.4,fill=black,label={\tiny$f$}] (f) at (1.25,0) {};
        \node[circle,scale=0.4,fill=black] (u) at (0,-0.6) {};
	\node[circle,scale=0.4,fill=black] (v) at (1.25,-0.6) {};
        \path[thick,purple] (e) edge [bend left=30] (f);
        \path[dashed, thick,blue] (e) edge (u);
	\path[dashed, thick,blue] (f) edge  (v);
    \end{tikzpicture}} 
\\ \noalign{\medskip} \hline 
             $\sfY_0(e) = f\neq \sfX_0(e)$ \qquad\
             \raisebox{-10pt}{ \begin{tikzpicture}
    \node[circle,scale=0.4,fill=black,label={\tiny$e$}] (e) at (0,0) {};
    \node[circle,scale=0.4,fill=black,label={\tiny$f$}] (f) at (1.25,0) {};
        \node[circle,scale=0.4,fill=black] (u) at (0,-0.6) {};
	\node[circle,scale=0.4,fill=black] (v) at (1.25,-0.6) {};
        \path[thick,dashed,blue] (e) edge [bend left=30] (f);
        \path[thick,purple] (e) edge (u);
	\path[thick,purple] (f) edge  (v);
    \end{tikzpicture}}
  \end{tabular}}
 & \begin{tabular}{c}{\small$~~^{(\star)}$}\\ $\geq 2$\end{tabular} 
 & \begin{tabular}{c}\\ $0$\end{tabular}  \\
\noalign{\smallskip}
\midrule[0.25pt]
\raisebox{-1pt}{V}&
 {\footnotesize 
 \!\!\begin{tabular}{l}
 $\sfX_0(e) = \sfY_0(f)$\\
  $\sfX_0(f) = \sfY_0(e)$
  \end{tabular}
  \qquad\qquad
\raisebox{-12pt}{ \begin{tikzpicture}
    \node[circle,scale=0.4,fill=black,label={\tiny$e$}] (e) at (0,0) {};
    \node[circle,scale=0.4,fill=black,label={\tiny$f$}] (f) at (1.25,0) {};
        \node[circle,scale=0.4,fill=black] (u) at (0,-0.75) {};
	\node[circle,scale=0.4,fill=black] (v) at (1.25,-0.75) {};
        \path[thick,dashed,blue] (e) edge (u);
        \path[thick,purple] (e) edge (v);
	\path[thick,purple] (f) edge  (u);
	\path[thick,dashed,blue] (f) edge (v);
    \end{tikzpicture}}
  }
& $4$ 
& $0$
\\
\noalign{\medskip}
\midrule[0.25pt]
\raisebox{-3pt}{VI}&
{\footnotesize
\!\! \!\!\!\!\begin{tabular}{l}
  \begin{tabular}{l}
   $\sfX_0(e) = \sfY_0(f)$\\ $\sfX_0(f) \neq \sfY_0(e)$
 \end{tabular}
\qquad\qquad\ \raisebox{-12pt}{\begin{tikzpicture}
    \node[circle,scale=0.4,fill=black,label={\tiny$e$}] (e) at (0.375,0) {};
    \node[circle,scale=0.4,fill=black,label={\tiny$f$}] (f) at (1.125,0) {};
        \node[circle,scale=0.4,fill=black] (u) at (0,-0.75) {};
	\node[circle,scale=0.4,fill=black] (v) at (0.75,-0.75) {};
	\node[circle,scale=0.4,fill=black] (w) at (1.5,-0.75) {};
	\path[dashed, thick,blue] (e) edge  (u);
        \path[thick,purple] (e) edge (v);
	\path[dashed, thick,blue] (f) edge  (v);
	\path[thick,purple] (f) edge  (w);
    \end{tikzpicture}} 
 \\ 
 \noalign{\medskip}
 \hline
  \begin{tabular}{l} $\sfX_0(f) = \sfY_0(e)$\\ $\sfX_0(e) \neq \sfY_0(f)$
	\end{tabular}
\qquad\qquad\ \raisebox{-12pt}{\begin{tikzpicture}
    \node[circle,scale=0.4,fill=black,label={\tiny$e$}] (e) at (0.375,0) {};
    \node[circle,scale=0.4,fill=black,label={\tiny$f$}] (f) at (1.125,0) {};
        \node[circle,scale=0.4,fill=black] (u) at (0,-0.75) {};
	\node[circle,scale=0.4,fill=black] (v) at (0.75,-0.75) {};
	\node[circle,scale=0.4,fill=black] (w) at (1.5,-0.75) {};
	\path[dashed, thick,blue] (e) edge  (v);
        \path[thick,purple] (e) edge (u);
	\path[dashed, thick,blue] (f) edge  (w);
	\path[thick,purple] (f) edge  (v);
    \end{tikzpicture}} 
            \end{tabular}}
& $2$ 
& $0$
\\
\noalign{\medskip}
\midrule[0.25pt]
VII&
{\footnotesize \!\!\! \begin{tabular}{l} $e,\sfX_0(e),\sfY_0(e),$ \\
$f, \sfX_0(f),\sfY_0(f)$\\
        are all distinct\end{tabular}}
\qquad\quad\ \raisebox{-12pt}{\begin{tikzpicture}
    \node[circle,scale=0.4,fill=black,label={\tiny$e$}] (e) at (0.25,0) {};
    \node[circle,scale=0.4,fill=black,label={\tiny$f$}] (f) at (1.25,0) {};
        \node[circle,scale=0.4,fill=black] (u) at (0,-0.75) {};
	\node[circle,scale=0.4,fill=black] (v) at (0.5,-0.75) {};
	\node[circle,scale=0.4,fill=black] (w) at (1,-0.75) {};
	\node[circle,scale=0.4,fill=black] (x) at (1.5,-0.75) {};
        \path[thick,purple] (e) edge (u);
	\path[thick,purple] (f) edge  (w);
	\path[dashed, thick,blue] (e) edge  (v);
	\path[dashed, thick,blue] (f) edge  (x);
    \end{tikzpicture}} 
& \begin{tabular}{c}{\small$~~^{(\star)}$}\\ $\geq 2$\end{tabular} 
& \begin{tabular}{c}{\small$~~^{(\star)}$}\\ $\geq 0$\end{tabular} 
\\
\noalign{\medskip}
\midrule[0.25pt]
\end{tabular}
\caption{Analysis of the change in the size of the common part of the two  matchings after one step of the coupling. In cases marked by  $^{(\star)}$, the difference could be larger if 
$
    \sfX_0(\sfY_0(e)) = \sfY_0(f) $
or $\sfY_0(\sfX_0(e)) = \sfX_0(f)
    $.}
\label{tab:coupling}
\end{table}

The corresponding value of $|\sfC_1|-|\sfC_0|$ in each of these cases are given in Table~\ref{tab:coupling}, from which it follows that under an invariant measure $\mu$, 
\begin{equation}\label{eq:Y-bd}
   0 =  \EE[|\sfC_1|-|\sfC_0|] \ge 2\, \PP(e,f\notin \sfC_0) - 4\, \PP(\textup{coupling fails})\,.
\end{equation}
For the first term on the \abbr{rhs} of \eqref{eq:Y-bd},
\begin{align}\label{eq:1st-term-in-Y-bd}
    \PP(e,f\notin \sfC_0 \mid \sfC_0) &= 
    \left(\frac{\sfm_n-|\sfC_0|}{\sfm_n}\right)
    \left(\frac{\sfm_n-|\sfC_0|-1}{\sfm_n-1}\right) 
    \,.
\end{align}
Combining these, we get that the \abbr{lhs} of~\eqref{e:deltaEn} is at most $2\sfm_n (\sfm_n-1) \,\PP(\textup{coupling fails})$.

For the latter, note that the coupling fails only under one of the following scenarios:
\begin{enumerate}[(a)]
    \item \emph{introducing a loop}: $v(e)=v(f)$ or $v(\sfY_0(e))=v(\sfY_0(f))$;
    \item \emph{introducing multiple edges}: $v(e)$ is connected to $v(f)$ in $\sfY_0 \setminus \{(e,\sfY_0(e)),(f,\sfY_0(f))\}$, or 
$v(\sfY_0(e))$ is connected to $v(\sfY_0(f))$ in $\sfY_0  \setminus \{(e,\sfY_0(e)),(f,\sfY_0(f))\}$.
\end{enumerate}
As $(\sfY_0(e),\sfY_0(f))$ has the same (uniform) distribution as $(e,f)$, we thus deduce that
\[
    \frac{1}{2} \PP(\textup{coupling fails})
    \le
     \PP(v(e)=v(f)) + \PP(v(e) \textup{ connected to }v(f)) \,.
\]
With $q_{ij}$ denoting the probability that $i \ne j$ are adjacent in $\sfY_0$, clearly
\[
\PP(v(e)\textup{ connected to }v(f)) = \sum_{i \ne j}  \frac{(D_i-1)(D_j-1)}{\sfm_n(\sfm_n-1)}q_{ij}\,.
\]
Similarly, recalling that $\sum_{j} q_{ij} = D_i$ for any $i \in [n]$, we have that
\[
\PP(v(e)=v(f))=\sum_{i=1}^n \frac{D_i(D_i-1)}{\sfm_n(\sfm_n-1)} = \sum_{i \ne j} 
\frac{(D_i+D_j)/2-1}{\sfm_n(\sfm_n-1)} q_{ij} \,.
\]
Adding these expressions and reducing the sum by symmetry to $j>i$, we arrive at
\begin{equation}\label{e:bd-coupling}
\frac{1}{2} \PP(\textup{coupling fails})
    \le \frac{1}{\sfm_n(\sfm_n-1)} \sum_{i=1}^{n-1} \sum_{j=i+1}^n  (2D_i D_j-D_i - D_j) q_{ij}\,.
\end{equation}
With $j \mapsto D_j$ non-decreasing and $\sum_{j>i} q_{ij} \le D_i$, by
replacing $q_{ij}$ with $\one_{\{j \le i+D_i\}}$ we upper bound the \abbr{rhs} 
of \eqref{e:bd-coupling}. 
Combining this with~\eqref{eq:Y-bd}--\eqref{eq:1st-term-in-Y-bd} establishes \eqref{e:deltaEn}, 
thereby 
concluding the proof of Proposition~\ref{prop:coupling}. \qed

\section{Analysis of the limiting density}\label{sec:dens}

\begin{remark}\label{rem:pm}
With $\nu^{(2)}$ denoting the pushforward of $\nu$ by the map $x \mapsto x^2$
(that is, the weak limit of $\cL^{\bLambda_n^2}$), we have similarly to Remark~\ref{rem:goe} that
$\mu_{\textsc{mp}} = \nu^{(2)} \boxtimes \sigma_{\semic}^{(2)}$,
where the pushforward $\sigma_{\semic}^{(2)}$
(of density $(2\pi)^{-1} \sqrt{4/x-1}$ on $[0,4]$),
is the limiting empirical distribution of
singular values of $n^{-1/2} \widetilde \bX_n$.
\end{remark}

\begin{proof}[\emph{\textbf{Proof of Proposition~\ref{lem-mu-Stil}}}]
The matrix $\bM_n := n^{-1} \widetilde \bX_n \bLambda_n^2 \widetilde \bX_n^\star$
has the same \abbr{esd} as
$n^{-1} \bLambda_n \widetilde \bX_n \widetilde \bX_n^\star \bLambda_n$.
Thus, $\mu_{\textsc{mp}}$ is also the limiting \abbr{esd}
for $\bM_n$ (see \cite{MaPa67,SiBa95}). Taking
$\cL^{\bLambda_n} \Rightarrow \nu$ with $d \nu/d\nu_{\hat D} (x) = x$
yields the Cauchy--Stieltjes transform $G_{\mu_{\textsc{mp}}}(z) = h(z)$
which is the unique decaying
to zero as $|z| \to \infty$, $\CC_+$-valued analytic on $\CC_+$, solution of
\begin{align}\label{eq-h(z)}
h =  \left(\EE\left[\tfrac{\hat D^2}{1 + h\hat D}\right]-z\right)^{-1}
=  - z^{-1} \EE \left[\tfrac{\hat D}{1 + h \hat D}\right] \,.
\end{align}
Indeed, the \abbr{lhs} of \eqref{eq-h(z)} merely re-writes
the fact that $\xi(\cdot)$ of \eqref{eq-h(z)-inv} is such that
$\xi(h(z))=z$ on $\CC_+$, while having
$\int x d \nu_{\hat D} = 1$, one thereby gets the
\abbr{rhs} of \eqref{eq-h(z)} by elementary algebra.
Recall \cite[Prop. 5(a)]{ArPe09} that the Cauchy--Stieltjes transform
of the symmetric measure $\widetilde \mu$
having the pushforward $\widetilde \mu^{(2)} =\mu_{\textsc{mp}}$
under the map $x \mapsto x^2$, is given for $\Re(z)>0$
by $g(z)=z h(z^2):\CC_+ \mapsto \CC_+$,
which by the \abbr{rhs} of \eqref{eq-h(z)} satisfies for $\Re(z)>0$,
\begin{align}\label{eq-g(z)}
g = - \EE\left[\frac{\hat D}{z + g \hat D}\right] \,.
\end{align}
By the symmetry of the measure $\widetilde \mu$ on $\RR$ we
know that $g(-\bar z)=-\bar g(z)$ thereby extending
the validity of \eqref{eq-g(z)} to all $z \in \CC_+$.
Applying the implicit function theorem in a suitable
neighborhood of $(-z^{-1},g)=(0,0)$ we
further deduce that $g(z)=G_{\widetilde \mu}(z)$ is
the unique $\CC_+$-valued, analytic on $\CC_+$ solution of
\eqref{eq-g(z)} tending to zero as $\Im(z) \to \infty$.
Recall the $S$-transform defined via~\eqref{dfn:m-trans}--\eqref{eq:S-trans} 
for $\varphi \ne \delta_0$ supported on $\RR_+$ and similarly
for symmetric measure $\psi$. In particular (see~\cite[Eqn. (20]{ArPe09}),
\[S_{\sigma_\semic}(w)=w^{-1/2}\,.\] 
Further, from~\eqref{dfn:m-trans} we see that
\eqref{eq-g(z)} results with $m_{\nu_{\hat D}} (-z^{-1} g) = g^2$,
yielding 
\[ S_{\nu_{\hat D}}(g^2)= - (1+g^{-2}) z^{-1} g\,.\]
Since
$S_\mu (w) = S_{\nu_{\hat D}}(w) S_{\sigma_{\semic}}(w)$,
we get
$S_\mu(g^2)= - (1+g^{-2}) z^{-1}$ and consequently $m_\mu(-z^{-1})=g^2$.
The latter amounts to
\begin{equation}\label{eq:f-form2}
f(z): = -z^{-1} (1+g^2) = \int \frac{1}{-t-z} d\mu(t) \,,
\end{equation}
which since $\mu$ is symmetric, matches the stated relation
$f(z) = G_\mu(z)$ of \eqref{eq:f-form}.
\end{proof}

\begin{proof}[\textbf{\emph{Proof of Proposition~\ref{prop:density}}}]
Recall from~\eqref{eq:f-form2} that
$f(z)=- z h(z^2)^2 - z^{-1}$ for $z \in \CC_+$ and $\Re(z)>0$.
When $z \to x \in (0,\infty)$ we further have that
$h(z^2) \to h(x^2)$ and hence
\begin{equation}\label{eq:plemelj}
\frac{1}{\pi} \Im(f(z)) \to - \frac{1}{\pi} \Im(x h(x^2)^2)
= - 2 \Re(h(x^2))
\widetilde \brho(x) \,,
\end{equation}
where the last identity is due to \eqref{dfn:pm-density}.
Thus, for a.e. $x > 0$
the density $\brho(x)$
exists
and given by Plemelj formula, namely the
\abbr{rhs} of \eqref{eq:plemelj}. The continuity
of $x \mapsto h(x)$ implies the same for the
symmetric density $\brho(x)$, thereby we
deduce the validity of \eqref{dfn:rho} at \emph{every} $x \ne 0$.
While proving \cite[Thm. 1.1]{SiCh95} it was shown
that $h(z)$ extends analytically around each $x \in \RR \setminus \{0\}$
where $\Im(h(x))>0$ (see also Remark~\ref{rem:pm-support}).
In particular, \eqref{dfn:rho} implies that
$\brho(x)$ is real analytic at any $x \ne 0$ where it is positive.
Further, in view of \eqref{dfn:rho}, the support identity
$\supp(mu)=\supp(\widetilde \mu)$ is an immediate consequence of
having $\Re(h(x))<0$ for all $x>0$ (as
shown in Lemma~\ref{lem:g-bdd}). Similarly,
the stated relation with $\supp(\mu_{\textsc{mp}})$ follows from
the explicit relation $\widetilde \brho (x) = |x| \brho_{\textsc{mp}}(x^2)$.
Finally, Lemma~\ref{lem:g-bdd} provides the stated bounds
on $\widetilde \brho$ and $\brho$
(see \eqref{eq:g-bdd} and \eqref{eq:rho-bdd}, respectively),
while showing that if $\nu_{\hat D}(\{0\})=0$ then $\mu$ is absolutely
continuous.
\end{proof}

Our next lemma provides the estimates we deferred when
proving Proposition~\ref{prop:density}.
\begin{lemma}\label{lem:g-bdd}
The function $g(z)=G_{\widetilde \mu}(z)$ satisfies
\begin{equation}\label{eq:g-bdd}
|g(z)| \leq 1 \wedge \frac{2}{|\Re(z)|} \,, \qquad \forall z \in \CC_+ \cup \RR
\end{equation}
and \eqref{eq-g(z)} holds for $z \in \CC_+ \cup \RR \setminus \{0\}$,
resulting with $\Re(h(x))<0$ for $x>0$. In addition
\begin{equation}\label{eq:rho-bdd}
\brho(x) \le \frac{1}{\pi} \big( \EE \hat D^{-2})^{1/2} \wedge 4 |x|^{-3} \big)
\qquad \forall x \in \RR \,,
\end{equation}
and if $\nu_{\hat D}(\{0\})=0$, then
$\mu(\{0\})=0$.
\end{lemma}
\begin{proof} As explained when proving Proposition~\ref{lem-mu-Stil},
by the symmetry of $\widetilde \mu$, we only need to consider $\Re(z) \ge 0$.
Starting with $z \in \CC_+$, let
\begin{align*}
z    &=x+i\eta \qquad \mbox{ for $x \ge 0$ and } \eta > 0 \,,\\
g(z) &= - y + i \gamma \, \quad\mbox{ for $y\in \RR$ and }\gamma>0\,.
\end{align*}
Then, separating the real and imaginary parts of~\eqref{eq-g(z)} gives
\begin{align}\label{eq-y-gamma}
y = \EE\left[ \hat D (x-y\hat D)\hat W^{-2} \right]\,,\quad &
\gamma = \EE\left[\hat D(\eta + \gamma \hat D) \hat W^{-2}\right]\,,
\end{align}
where $\hat W := |z + g(z) \hat D|$
must be a.s.\ strictly positive (or else $\gamma=\infty$).
Next, defining
\begin{equation}\label{eq:AB-def}
A = A(z) := \EE[\hat D {\hat W}^{-2}]\,,\qquad
B = B(z) := \EE[{\hat D}^2 {\hat W}^{-2}]\,,
\end{equation}
both of which are positive and finite
(or else $\gamma=\infty$),
translates \eqref{eq-y-gamma} into
\begin{align*}
y =  A x - B y\,,\qquad &
\gamma = A \eta  + B \gamma\,.
\end{align*}
Therefore,
\begin{equation}
\label{eq-y-gamma-sol}
y = \frac{A x}{1+B}\,,\qquad \gamma = \frac{A \eta}{1-B}\,.
\end{equation}
Since $\gamma>0$, necessarily $0<B<1$ and $y \ge 0$ is strictly
positive iff $x>0$. Next, by \eqref{eq-g(z)}, Jensen's inequality
and \eqref{eq:AB-def},
\begin{align}
\label{eq:bd-W}
|g|
\le \EE \Big[ \hat D \hat W^{-1} \Big] := V(z) \le \sqrt{B} \le 1 \,.
\end{align}
Further, letting $D \sim \nu$ be the size-biasing of $\hat D$
and $W:=|z+g(z) D|$, we have that
\begin{equation}\label{eq:size-bias}
g(z)=-\EE[(z+g(z)D)^{-1}]\,,\;\;\;
V=\EE[W^{-1}]\,, \;\;\;
A=\EE[W^{-2}]\,.
\end{equation}
With $B < 1$ we thus have by \eqref{eq-y-gamma-sol},
\eqref{eq:size-bias} and Jensen's inequality, that
\[
\frac{|x| A}{2} \le \frac{|x|A}{1+B} = |y| \le |g| \le V \le
\sqrt{A} \,.
\]
Consequently, $|g(z)| \le \sqrt{A} \le 2/|x|$ as claimed.
Next, recall \cite[Theorem~1.1]{SiCh95} that $h(z) \to h(x)$
whenever $z \to x \ne 0$, hence same applies to $g(\cdot)$
with \eqref{eq:g-bdd} and the bound $B(z) \le 1$,
also applicable throughout $\RR \setminus \{0\}$.
Further, having $z_n \to x \ne 0$ implies that $|\Re(z_n)|$ is bounded
away from zero, hence $\{A(z_n)\}
$ are uniformly bounded.
In view of \eqref{eq:size-bias}, this yields the uniform integrability
of $(z_n+g(z_n) D)^{-1}$ and thereby its $L_1$-convergence
to the absolutely-integrable $(x+g(x) D)^{-1}$. Appealing
to the representation \eqref{eq:size-bias} of $g(z)$
we conclude that \eqref{eq-g(z)} extends to $\RR \setminus \{0\}$.
Utilizing \eqref{eq-g(z)} at $z=x > 0$ we see that
$0 < |g(x)|^2 \le A(x)$
due to \eqref{eq:size-bias}.
Hence, from \eqref{eq-y-gamma} we have as claimed,
\[
\Re(h(x^2)) =  x^{-1} \Re(g(x)) = \frac{-A(x)}{1+B(x)} < 0 \,.
\]
From \eqref{eq-y-gamma-sol} we have that $g(z)=i \gamma$
when $z=i\eta$,
where by \eqref{eq-g(z)}, for any $\delta>0$,
\[
\gamma=\EE\Big[\frac{\hat D}{\eta + \gamma \hat D}\Big]
\ge \frac{\delta}{\eta + \gamma \delta} \nu_{\hat D}([\delta,\infty)) \,.
\]
Taking $\eta \downarrow 0$ followed by $\delta \downarrow 0$ we see that
$\gamma(i \eta) \to \gamma(0) = 1$, provided $\nu_{\hat D}(\{0\})=0$.
By definition of the Cauchy--Stieltjes transform
and bounded convergence, we have then
\[
\mu(\{0\})=-\lim_{\eta \downarrow 0} \Re(i \eta f(i \eta))
= 1 - [\lim_{\eta \downarrow 0} \gamma(i \eta) ]^2 = 0 \,,
\]
due to \eqref{eq:f-form2} (and having $\Re(g(i \eta))=0$).
Finally, from \eqref{eq-g(z)} and the \abbr{lhs} of \eqref{eq:f-form2}
we have that $f(z) = -\EE[(z+ g(z) \hat D)^{-1}]$ throughout $\CC_+$, hence
by Cauchy--Schwarz
\[
|f(z)| \le \EE[\hat W^{-1}]
\le \sqrt{B(z) \EE [\hat D^{-2}]} \le \EE[\hat D^{-2}]^{1/2}
\]
is uniformly bounded when $\EE \hat D^{-2}$ is finite.
Up to factor $\pi^{-1}$ this yields the stated uniform bound on $\brho(x)$,
namely the \abbr{rhs} of \eqref{eq:plemelj}. At any $x>0$ the latter
is bounded above also by $\frac{1}{\pi x} |g(x)|^2$, with \eqref{eq:rho-bdd}
thus a consequence of \eqref{eq:g-bdd}.
\end{proof}

\begin{proof}
[\emph{\textbf{Proof of Corollary~\ref{cor:two-atoms}}}]
Fixing $\alpha>\eta>0$ we have that
\[
\nu_{\hat D} (\{\alpha\})=q_o\,,\quad \nu_{\hat D} (\{\eta\})=1-q_o
\]
and since $1 = \EE \hat D = \alpha q_o + \eta(1-q_o)$, further
$\alpha>1>\eta$.
By Remark~\ref{rem:pm-support} we identify $\supp(\mu)$ upon
examining the regions in which $\xi'(-v)>0$ for $\RR$-valued
$v \notin \{0,\alpha^{-1},\eta^{-1}\}$.
Since $\Re(h(x))<0$ for $x>0$ (see Lemma \ref{lem:g-bdd}),
for $\supp(\mu) \cap \RR_+$ it suffices to consider the sign of
\[
\xi'(-v) = \frac1{v^2} - \frac{q\alpha^2}{(1-v\alpha)^2} - \frac{(1-q)\eta^2}{(1-v\eta)^2}\,,
\]
when $v \in (0,\infty) \setminus \{ \alpha^{-1},\eta^{-1}\}$ and
$q := \alpha q_o$.
Observe that $\xi'(-v)>0$ for such $v$ iff
\begin{align*}
P(v) & := av^3+bv^2+cv+d  \\
&= -2\alpha\eta(q\eta+(1-q)\alpha) v^3 + \left(q\eta^2 + 4\alpha\eta + (1-q)\alpha^2\right) v^2 - 2(\alpha + \eta) v + 1  > 0 \,.
\end{align*}
Noting that $\lim_{v\to\infty} P(v)=-\infty$ and $\lim_{v \downarrow 0}P(v)=1$, we infer from Remark~\ref{rem:pm-support} that $\supp(\mu)$ has holes iff $P(v)$ has three distinct positive roots.
As Descrates' rule of signs is satisfied ($a,c<0$ and $b,d>0$), the latter occurs iff the discriminant $\sfD(P)$ is positive. Evaluating $\sfD(P)$ shows that
\[  \sfD(P) = b^2 c^2 - 4 a c^3 - 4b^3 d + 18abcd - 27a^2 d^2 =  4 q (1-q) (\alpha-\eta)^2 \big( \alpha \phi
- q \theta \big) \,,
\]
where
\begin{align*}
  \theta := (\alpha - \eta)(\alpha + \eta)^3\,,\qquad
  \phi := (\alpha - 2 \eta)^3 \,.
\end{align*}
Having $q=\alpha q_o$ and $\theta>0$ we conclude that
$\sfD(P)>0$ iff $\phi/\theta > q_o$. That is
\[
\frac{\phi}{\theta} = \frac{(\alpha-2\eta)^3}
{(\alpha - \eta)(\alpha + \eta)^3} >
\frac{1-\eta}{\alpha-\eta} = q_0 \,.
\]
For $\varphi:=3\eta/(\alpha+\eta)$ and $\eta \in (0,1)$ this
translates into $1-\varphi > (1-\eta)^{1/3}$, or equivalently
\[
\frac{\alpha}{\eta} + 1 = \frac{3}{\varphi} >
\frac{3}{1-(1-\eta)^{1/3}} \,,
\]
as stated in \eqref{eq-support-holes}.
\end{proof}

\subsection*{Acknowledgment} The authors wish to thank Nick Cook, Alice Guionnet, Allan Sly 
and Ofer Zeitouni for many helpful discussions.
A.D.\ was supported in part by NSF grant DMS-1613091.
E.L.\ was supported in part by NSF grants DMS-1513403 and DMS-1812095.


\bibliographystyle{abbrv}
\bibliography{sparse_density}

\end{document}